\documentclass[a4paper,10pt]{article}

\usepackage[utf8]{inputenc}
\usepackage[T1]{fontenc}
\usepackage{lmodern}
\usepackage{standalone}
\usepackage{textgreek}
\usepackage{hyphenat}
\usepackage{microtype}
\usepackage{geometry}
\usepackage{caption}
\usepackage{subcaption}
\usepackage[colorlinks,unicode]{hyperref}
\usepackage{xspace}
\usepackage{float}
\usepackage{appendix}
\usepackage{tikz}
\usepackage{pgfplots}\pgfplotsset{compat=1.18}
\usepackage{amsthm}
\usepackage{amsfonts}
\usepackage{amssymb}
\usepackage{mathtools}
\usepackage{braket}
\usepackage{siunitx}
\usepackage{enumitem}
\usepackage{algorithm}
\usepackage[noend]{algpseudocode}
\usepackage{array}
\usepackage{booktabs}
\usepackage{blkarray}
\usepackage{cleveref}
\usepackage{orcidlink}
\usepackage{doi}

\hypersetup{
    breaklinks=true,
    linkcolor=blue,
    citecolor=blue,
    ,urlcolor=blue
}

\usepackage[backend=biber, style=numeric, sorting=none, maxbibnames=99, maxcitenames=99, doi=true, eprint=true, maxalphanames=3, maxcitenames=3, minalphanames=3, minnames=1]{biblatex}
\usetikzlibrary{cd}

\setlist[enumerate,1]{label={(\arabic*)}}

\makeatletter
\newcounter{algorithmicH}
\let\oldalgorithmic\algorithmic
\renewcommand{\algorithmic}{%
    \stepcounter{algorithmicH}
    \oldalgorithmic}
\renewcommand{\theHALG@line}{ALG@line.\thealgorithmicH.\arabic{ALG@line}}
\makeatother

\allowdisplaybreaks

\theoremstyle{plain}
\newtheorem{thm}{Theorem}
\newtheorem{lem}[thm]{Lemma}
\newtheorem{cor}[thm]{Corollary}
\newtheorem{prop}[thm]{Proposition}

\newtheorem{defn}[thm]{Definition}

\theoremstyle{definition}

\newtheorem{ex}[thm]{Example}
\theoremstyle{plain}

\crefname{thm}{theorem}{theorems}
\Crefname{thm}{Theorem}{Theorems}
\crefname{defn}{definition}{definitions}
\Crefname{defn}{Definition}{Definitions}
\crefname{prop}{proposition}{propositions}
\Crefname{prop}{Proposition}{Propositions}
\crefname{lem}{lemma}{lemmas}
\Crefname{lem}{Lemma}{Lemmas}
\crefname{cor}{corollary}{corollaries}
\Crefname{cor}{Corollary}{Corollaries}
\crefname{ex}{example}{examples}
\Crefname{ex}{Example}{Examples}
\crefname{rem}{remark}{remarks}
\Crefname{rem}{Remark}{Remarks}
\crefname{hypot}{hypothesis}{hypotheses}
\Crefname{hypot}{Hypothesis}{Hypotheses}
\crefname{conj}{conjecture}{conjectures}
\Crefname{conj}{Conjecture}{Conjectures}

\makeatletter
\newcommand{\subalign}[1]{%
  \vcenter{%
    \Let@ \restore@math@cr \default@tag
    \baselineskip\fontdimen10 \scriptfont\tw@
    \advance\baselineskip\fontdimen12 \scriptfont\tw@
    \lineskip\thr@@\fontdimen8 \scriptfont\thr@@
    \lineskiplimit\lineskip
    \ialign{\hfil$\m@th\scriptstyle##$&$\m@th\scriptstyle{}##$\hfil\crcr
      #1\crcr
    }%
  }%
}
\makeatother

\newcolumntype{P}[1]{>{\centering\arraybackslash}p{#1}}
\newcolumntype{M}[1]{>{\centering\arraybackslash}m{#1}}

\newcommand{\F}{\mathbb{F}}                                         
\newcommand{\Fq}{\F_{q}}                                            
\newcommand{\Fqn}{\Fq^{n}}                                          
\newcommand{\Fqnxn}{\Fq^{n \times n}}                               
\newcommand{\Fqx}{\Fq^{\times}}                                     

\newcommand{\abs}[1]{\left\vert #1 \right\vert}                     
\DeclareMathOperator{\conv}{conv}                                   
\newcommand{\floor}[1]{\left\lfloor #1 \right\rfloor}               
\DeclareMathOperator{\Hom}{Hom}                                     
\DeclareMathOperator{\lcm}{lcm}                                     
\newcommand{\R}{\mathbb{R}}                                         
\ifdefined\widebar
\else
    \newcommand\widebar[1]{\mathop{\overline{#1}}}                  
\fi
\newcommand{\Z}{\mathbb{Z}}                                         
\newcommand{\Zn}{\Z^{n}}                                            

\newcommand{\AffSp}[2]{\mathbb{A}^{#1}_{#2}}                        
\newcommand{\ProjSp}[2]{\mathbb{P}^{#1}_{#2}}                       

\newcommand{\dehom}{\text{\normalfont deh}}                         
\newcommand{\homog}{\text{\normalfont hom}}                         
\newcommand{\topcomp}{\text{\normalfont top}}                       

\newcommand{\degree}[1]{\deg \left( #1 \right)}                    
\DeclareMathOperator{\LC}{LC}                                       
\DeclareMathOperator{\LM}{LM}                                       

\DeclareMathOperator{\Spec}{Spec}                                   

\newcommand{\Anemoi}{\texttt{Anemoi}\xspace}

\newcommand{\Chaghri}{\textsc{Chaghri}\xspace}
\newcommand{\Ciminion}{\texttt{Ciminion}\xspace}

\newcommand{\GMiMC}{\texttt{GMiMC}\xspace}

\newcommand{\GMiMCHash}{\texttt{GMiMCHash}\xspace}

\newcommand{\Griffin}{\textsc{Griffin}\xspace}

\newcommand{\Hades}{\textsc{Hades}\xspace}
\newcommand{\HERA}{\textsf{HERA}\xspace}
\newcommand{\Hydra}{\textsf{Hydra}\xspace}

\newcommand{\MiMC}{\texttt{MiMC}\xspace}
\newcommand{\Monolith}{\texttt{Monolith}\xspace}

\newcommand{\Poseidon}{\textsc{Poseidon}\xspace}

\newcommand{\SageMath}{\texttt{SageMath}\xspace}

\definecolor{maroon}{cmyk}{0, 0.87, 0.68, 0.32}
\definecolor{halfgray}{gray}{0.55}
\definecolor{ipython_frame}{RGB}{207, 207, 207}
\definecolor{ipython_bg}{RGB}{247, 247, 247}
\definecolor{ipython_red}{RGB}{186, 33, 33}
\definecolor{ipython_green}{RGB}{0, 128, 0}
\definecolor{ipython_cyan}{RGB}{64, 128, 128}
\definecolor{ipython_purple}{RGB}{170, 34, 255}

\usepackage{listings}
\newlength{\MaxSizeOfLineNumbers}%
\lstdefinelanguage{MyPython}{
    morekeywords={access,and,break,class,continue,def,del,elif,else,except,exec,finally,for,from,global,if,import,in,is,lambda,not,or,pass,print,raise,return,try,while},%
    morekeywords=[2]{abs,all,any,basestring,bin,bool,bytearray,callable,chr,classmethod,cmp,compile,complex,delattr,dict,dir,divmod,enumerate,eval,execfile,file,filter,float,format,frozenset,getattr,globals,hasattr,hash,help,hex,id,input,int,isinstance,issubclass,iter,len,list,locals,long,map,max,memoryview,min,next,object,oct,open,ord,pow,property,range,raw_input,reduce,reload,repr,reversed,round,set,setattr,slice,sorted,staticmethod,str,sum,super,tuple,type,unichr,unicode,vars,xrange,zip,apply,buffer,coerce,intern},%
    sensitive=true,%
    morecomment=[l]\#,%
    morestring=[b]',%
    morestring=[b]",%
    morestring=[s]{'''}{'''},
    morestring=[s]{"""}{"""},
    morestring=[s]{r'}{'},
    morestring=[s]{r"}{"},%
    morestring=[s]{r'''}{'''},%
    morestring=[s]{r"""}{"""},%
    morestring=[s]{u'}{'},
    morestring=[s]{u"}{"},%
    morestring=[s]{u'''}{'''},%
    morestring=[s]{u"""}{"""},%
    literate=
    {á}{{\'a}}1 {é}{{\'e}}1 {í}{{\'i}}1 {ó}{{\'o}}1 {ú}{{\'u}}1
    {Á}{{\'A}}1 {É}{{\'E}}1 {Í}{{\'I}}1 {Ó}{{\'O}}1 {Ú}{{\'U}}1
    {à}{{\`a}}1 {è}{{\`e}}1 {ì}{{\`i}}1 {ò}{{\`o}}1 {ù}{{\`u}}1
    {À}{{\`A}}1 {È}{{\'E}}1 {Ì}{{\`I}}1 {Ò}{{\`O}}1 {Ù}{{\`U}}1
    {ä}{{\"a}}1 {ë}{{\"e}}1 {ï}{{\"i}}1 {ö}{{\"o}}1 {ü}{{\"u}}1
    {Ä}{{\"A}}1 {Ë}{{\"E}}1 {Ï}{{\"I}}1 {Ö}{{\"O}}1 {Ü}{{\"U}}1
    {â}{{\^a}}1 {ê}{{\^e}}1 {î}{{\^i}}1 {ô}{{\^o}}1 {û}{{\^u}}1
    {Â}{{\^A}}1 {Ê}{{\^E}}1 {Î}{{\^I}}1 {Ô}{{\^O}}1 {Û}{{\^U}}1
    {œ}{{\oe}}1 {Œ}{{\OE}}1 {æ}{{\ae}}1 {Æ}{{\AE}}1 {ß}{{\ss}}1
    {ç}{{\c c}}1 {Ç}{{\c C}}1 {ø}{{\o}}1 {å}{{\r a}}1 {Å}{{\r A}}1
    {€}{{\EUR}}1 {£}{{\pounds}}1
    {^}{{{\color{ipython_purple}\^{}}}}1
    {=}{{{\color{ipython_purple}=}}}1
    {+}{{{\color{ipython_purple}+}}}1
    {*}{{{\color{ipython_purple}$^\ast$}}}1
    {/}{{{\color{ipython_purple}/}}}1
    {+=}{{{+=}}}1
    {-=}{{{-=}}}1
    {*=}{{{$^\ast$=}}}1
    {/=}{{{/=}}}1,
    literate=
    *{-}{{{\color{ipython_purple}-}}}1
    {?}{{{\color{ipython_purple}?}}}1,
    identifierstyle=\color{black}\ttfamily,
    commentstyle=\color{ipython_cyan}\ttfamily,
    stringstyle=\color{ipython_red}\ttfamily,
    keepspaces=true,
    showspaces=false,
    showstringspaces=false,
    basicstyle=\scriptsize,
    keywordstyle=\color{ipython_green}\ttfamily,
}

\title{A Degree Bound for the c-Boomerang Uniformity}

\author{
    Matthias Johann Steiner \orcidlink{0000-0001-5206-6579} \\
    Alpen-Adria-Universit\"at Klagenfurt, \\
    Universit\"atsstraße 65-67, 9020 Klagenfurt am W\"orthersee, Austria \\
    \href{mailto:mattsteiner@edu.aau.at}{mattsteiner@edu.aau.at}}
\date{}

\begin{document}

    \maketitle

    \begin{abstract}
        Let $\mathbb{F}_q$ be a finite field, and let $F \in \Fq [X]$ be a polynomial with $d = \degree{F}$ such that $\gcd \left( d, q \right) = 1$.
        In this paper we prove that the $c$-Boomerang uniformity, $c \neq 0$, of $F$ is bounded by
        $
        \begin{rcases}
            \begin{cases}
                d^2, & c^2 \neq 1, \\
                d \cdot (d - 1), & c = - 1, \\
                d \cdot (d - 2), & c = 1
            \end{cases}
        \end{rcases}
        .$
        For all cases of $c$, we present tight examples for $F \in \Fq [X]$.

        Additionally, for the proof of $c = 1$ we establish that the bivariate polynomial $F (x) - F (y) + a \in k [x, y]$, where $k$ is a field of characteristic $p$ and $a \in k \setminus \{ 0 \}$, is absolutely irreducible if $p \nmid \deg \left( F \right)$.
    \end{abstract}

    \section{Introduction}
    Differential cryptanalysis introduced by Biham \& Shamir \cite{C:BihSha90} studies the propagation of input differences through a symmetric key cipher or a hash function.
    The susceptibility of a function $F: \Fq \to \Fq$ against differential cryptanalysis is expressed in the \emph{Differential Distribution Table} (DDT)
    \begin{equation}
        \delta_F (a, b) = \abs{\left\{ \mathbf{x} \in \Fqn \mid F (x + a) - F (x) = b \right\}},
    \end{equation}
    where $a, b \in \Fq$.
    The maximum entry with $a \neq 0$ of the DDT is also known as \emph{differential uniformity} \cite{EC:Nyberg93}.
    In a recent work Ellingsen et al.\ \cite{Ellingsen-cDifferential} generalized the DDT to $c$-scaled differences.

    A variant of differential cryptanalysis is the Boomerang attack introduced by Wagner \cite{FSE:Wagner99}.
    In the spirit of the DDT, Cid et al.\ \cite{EC:CHPSS18} introduced the \emph{Boomerang Connectivity Table} (BCT) to quantify the resistance of a permutation $F: \Fq \to \Fq$ against the Boomerang attack.
    Analogous to the $c$-scaled DDT, St\u{a}nic\u{a} \cite{Stanica-cBoomerang} introduced $c$-scaled differences to the BCT
    \begin{equation}\label{Equ: c-BCT definition}
        {}_{c} \mathcal{B}_F (a, b) = \abs{\left\{ (x, y) \in \Fq^2 \; \middle\vert \;
            \begin{aligned}
                F (x + y) &- c \cdot F (x) = b \\
                c \cdot F (x + y + a) &- F (x + a) = c \cdot b
            \end{aligned}
            \right\}},
    \end{equation}
    where $a, b \in \Fq$ and $c \in \Fqx$.
    The maximum entry of the $c$-BCT for $a, b \neq 0$ is also known as \emph{$c$-Boomerang uniformity}.

    Since its introduction, the $c$-BCT has been studied for many special permutation polynomials \cite{Stanica-cBoomerang,Stanica-SwappedInverse,Stanica-Characterization,Hasan-BinaryGold,Li-Boomerang-2024,Li-Boomerang-2025}.
    However, the mathematical investigations into functions with low ($c$-)Boomerang uniformity are in stark contrast with the actual requirements of many modern symmetric cryptography primitives.
    The past decade has seen a fast development of cryptographic primitives and protocols with privacy enhancing features, such as Fully-Homomorphic Encryption (FHE), Multi-Party Computation (MPC) and Zero-Knowledge (ZK) proofs.
    Performance of MPC, FHE \& ZK can often be improved via a dedicated symmetric cipher or hash function.
    We provide an (incomplete) overview of proposed symmetric primitives for these applications:
    \begin{itemize}
        \item Ciphers and Pseudo-Random Functions for MPC \cite{CCS:GRRSS16}: \MiMC \cite{AC:AGRRT16}, \GMiMC \cite{ESORICS:AGPRRR19}, \Hades \cite{EC:GLRRS20}, \Ciminion \cite{EC:DGGK21}, \Hydra \cite{EC:GOSW23}.

        \item Hash functions for ZK: \GMiMCHash \cite{ESORICS:AGPRRR19}, \Poseidon \cite{AFRICACRYPT:GraKhoSch23}, \Anemoi \cite{C:BBCPSV23}, \Griffin \cite{C:GHRSWW23}, \Monolith \cite{ToSC:GKLRSW24}.

        \item Hybrid FHE \cite{Naehring-Practical}: \Chaghri \cite{CCS:AshMahTop22}, \HERA \cite{AC:CHKLLL21}.
    \end{itemize}
    For efficiency in FHE/MPC/ZK applications these primitives must satisfy two novel design criteria: They must be native over a relatively large finite field $\Fq$, where $q \geq 2^{30}$ and often $q$ is prime, and their non-linear components must be representable with a low number of multiplications.
    The latter condition is often achieved via low-degree univariate permutations.

    Unfortunately, the aforementioned investigations into the ($c$-)Boomerang uniformity hardly yield any results applicable to cryptanalysis of FHE/MPC/ZK-friendly primitives.
    Henceforth, cryptodesigners would have to work out the ($c$-)BCT by hand for every novel S-Box.
    While this is in principle feasible, for example Wang et al.\ \cite{Wang-Boomerang} have worked out the $1$-Boomerang uniformity of permutation polynomials up to degree $6$, such an analysis quickly spans multiple pages for higher degrees.
    Recently, Steiner \cite{Steiner-Boomerang} addressed this gap in the literature for power permutations $x^d$ of $\Fq$ with $\gcd \left( d, q \right) = 1$.
    For $a, b \in \Fqx$ he proved that
    \begin{equation}\label{Equ: c-BCT monomials}
        {}_c \mathcal{B}_{x^d} (a, b) \leq
        \begin{cases}
            d^2, & c^2 \neq 1, \\
            d \cdot (d - 1), & c = - 1, \\
            d \cdot (d - 2), & c = 1.
        \end{cases}
    \end{equation}

    In this work we extend \Cref{Equ: c-BCT monomials} to permutation polynomials $F \in \Fq [x]$ under minimal structural assumptions on $F$.
    In particular, \Cref{Equ: c-BCT monomials} holds for any polynomial with $\gcd \big( \degree{F}, q \big) = 1$, see \Cref{Th: degree bound c = 1}, \Cref{Prop: c-boomerang unfiformity} and \Cref{Th: degree bound c = -1}.

    \subsection{Techniques to Prove the Main Result}
    On a high level Steiner's proof of \Cref{Equ: c-BCT monomials} proceeds in two steps.
    \begin{enumerate}
        \item Prove that one of the $c$-Boomerang polynomials in \Cref{Equ: c-BCT definition} is irreducible over the algebraic closure $\overline{\Fq}$ of $\Fq$, and verify that the other polynomial is not a multiple of the irreducible one.

        \item Homogenize the polynomials to interpret them as plane curves over $\ProjSp{2}{\overline{\Fq}}$, and estimate the intersection number with B\'ezout's theorem.
    \end{enumerate}
    For permutation monomials $x^d$ one $c$-Boomerang polynomial becomes $f_1 (x, z) = z^d - c \cdot x^d - b$, where $z = x + y$.
    Interpreting $f_1$ as univariate polynomial in $\Fq [x] \big[ z \big]$ one can prove its irreducibility via Eisenstein's criterion (if $\gcd \left( d, q \right) = 1$), see \cite[Lemma 3.3]{Steiner-Boomerang}.
    Unfortunately, this argument does not generalize to arbitrary permutation polynomials $F$ since it crucially relies on the fact that one can rewrite $f_1 (x, z) = z^d - (x - \alpha) \cdot g (x)$, where $\alpha \in \Fq$ is such that $c \cdot \alpha^d = -b$ and $g (\alpha) \neq 0$.

    We resolve this problem for the $1$-Boomerang uniformity via the Newton polytope method due to Gao \cite{Gao-Irreducibility}.
    Let $f = \sum_{c_i \neq 0} c_i \cdot x_1^{d_{i_1}} \cdots x_n^{d_{i_n}} \in k [x_1, \dots, x_n]$ be a polynomial, where $k$ is a field.
    For all non-zero coefficients $c_i$ we now collect the exponent vectors $\{ \mathbf{d}_i \}_i$ and interpret them as vectors in $\R^n$.
    The convex closure of $\{ \mathbf{d}_i \}_i$ is called the Newton polytope $\mathcal{N}_f$ of $f$.
    In particular, for $f, g \in K [x_1, \dots, x_n]$ one has $\mathcal{N}_{f \cdot g} = \mathcal{N}_f + \mathcal{N}_g$ \cite[\S 2]{Gao-Irreducibility}.
    Conversely, if the Newton polytope of $f$ cannot be rewritten as sum of polytopes (with positive integer vertices), then $f$ is absolutely irreducible over $k$.

    The Newton polytope of $F (x + y) - F (x) - b$, where $b \neq 0$, turns out to be a triangle, and if $\gcd \big( \degree{F}, q \big) = 1$ then the triangle is integrally indecomposable, see \Cref{Prop: irreducible polynomial}.
    This yields the generalization of \Cref{Equ: c-BCT monomials} for $c = 1$.

    Unfortunately, for $c \neq 1$ the polytope method is not applicable because the $c$-Boomerang polynomials have too much symmetry.
    We resolve this problem with Gr\"obner basis techniques.
    Let $f \in k [x_1, \dots, x_n]$ be a polynomial, then we can rewrite it as sum of homogeneous polynomials $f = f_d + f_{d - 1} + \ldots + f_0$.
    In particular, $f_d$ is called the homogeneous highest degree component of $f$.
    The highest degree components of the $c$-Boomerang polynomials admit relatively simple bivariate ideals.
    By computing the degree reverse lexicographic (DRL) Gr\"obner basis of the highest degree component ideal, we can already estimate the number of solutions for the $c$-Boomerang polynomials, which henceforth yields the other two cases of \Cref{Equ: c-BCT monomials}, see \Cref{Prop: c-boomerang unfiformity} and \Cref{Prop: highest degree components c = -1}.

    Moreover, with Gr\"obner bases of $c$-BCT polynomial systems we can construct tight examples for \Cref{Equ: c-BCT monomials}.

    \subsection{Structure of the Paper}
    In \Cref{Sec: Preliminaries} we introduce the mathematical preliminaries required of this work.
    In \Cref{Sec: Irreducibility of Boomerang Polynomials} we prove our $1$-Boomerang uniformity bounds via geometric methods, and \Cref{Sec: Arithmetic of Boomerang Polynomial Systems} we investigate the arithmetic of $c$-Boomerang polynomial systems to prove the other two cases.
    In \Cref{Sec: Constructing Tight Examples} we construct tight examples for our bounds.
    We finish this paper with a short discussion in \Cref{Sec: Discussion}.

    \section{Preliminaries}\label{Sec: Preliminaries}
    Let $k$ be a field, then we denote with $\bar{k}$ the algebraic closure of $k$.
    We denote the multiplicative group of a field with $k^\times = k \setminus \{ 0 \}$.
    For $q \in \Z$ a prime power we denote the finite field with $q$ many elements as $\Fq$.
    Elements of a vector space $\mathbf{a} \in k^n$ are denoted with small bold letters.

    Let $f \in k [x_1, \dots, x_n]$ be a non-zero polynomial, and let $x_0$ be a new variable.
    We denote the homogenization of $f$ with respect to $x_0$ as
    \begin{equation}
        f^\homog (x_0, \dots, x_n) = x_0^{\degree{f}} \cdot f \left( \frac{x_1}{x_0}, \dots, \frac{x_n}{x_0} \right) \in k [x_0, \dots, x_n].
    \end{equation}
    Conversely, let $F \in k [x_0, \dots, x_n]$ be homogeneous, then its dehomogenization with respect to $x_0$ is denoted as
    \begin{equation}
        F^\dehom (x_1, \dots, x_n) = F (1, x_1, \dots, x_n) \in k [x_1, \dots, x_n].
    \end{equation}

    For a commutative ring with unity $R$, an ideal $I \subset R$ is an additive subgroup which is multiplicatively closed.
    For $I \subset R$ an ideal, its radical is defined as $\sqrt{I} = \left\{ f \in R \mid \exists n \geq 1 \!: f^n \in I \right\}$.
    It is well-known that the radical is also an ideal.


    \subsection{\texorpdfstring{$c$}{c}-Boomerang Uniformity}
    First, we formally recall the $c$-Boomerang Connectivity Table and the $c$-Boomerang uniformity.
    \begin{defn}[{c-Boomerang Uniformity, \cite[Remark~1]{Stanica-cBoomerang}}]\label[defn]{Def: boomerang uniformity}
        Let $\Fq$ be a finite field, let $c \in \Fq^\times$, and let $F: \Fq \to \Fq$ be a function.
        \begin{enumerate}
            \item\label{Item: boomerang connectivity table} Let $a, b \in \Fq$.
            The entry of the $c$-Boomerang Connectivity Table ($c$-BCT) of $F$ at $(a, b)$ is defined as
            \begin{equation*}
                {}_{c} \mathcal{B}_F (a, b) = \abs{\left\{ (x, y) \in \Fq^2 \; \middle\vert \;
                    \begin{aligned}
                        F (x + y) &- c \cdot F (x) = b \\
                        c \cdot F (x + y + a) &- F (x + a) = c \cdot b
                    \end{aligned}
                    \right\}}.
            \end{equation*}

            \item The $c$-Boomerang uniformity of $F$ is defined as
            \begin{equation*}
                \beta_{F, c} = \max_{a, b \in \Fqx} {}_{c} \mathcal{B}_F (a, b).
            \end{equation*}
        \end{enumerate}
    \end{defn}

    Note that St\u{a}nic\u{a} originally defined the $c$-BCT only for permutations \cite[Definition~3]{Stanica-cBoomerang} as
    \begin{equation}\label{Equ: boomerang original definition}
        {}_{c} \mathcal{B}_F (a, b) = \abs{\left\{ x \in \Fq \; \middle\vert \; F^{-1} \left( c^{-1} \cdot F (x + a) + b \right) - F^{-1} \big( c \cdot F (x) + b \big) = a \right\}},
    \end{equation}
    which coincides for $c = 1$ with the original definition of the BCT \cite[Definition~3.1]{EC:CHPSS18}.
    By \cite[Theorem~4]{Stanica-cBoomerang} \Cref{Def: boomerang uniformity} \ref{Item: boomerang connectivity table} and \Cref{Equ: boomerang original definition} are equivalent for permutations, and in our context it is more convenient to work with \Cref{Def: boomerang uniformity} \ref{Item: boomerang connectivity table}.

    \subsection{Plane Curves}
    For a commutative ring $R$, the set of prime ideals $\Spec \left( R \right) = \left\{ \mathfrak{p} \subset R \mid \mathfrak{p} \text{ prime ideal} \right\}$ is called the spectrum of $R$.
    It is well-known that $\Spec \left( R \right)$ can be equipped with the Zariski topology, see for example \cite[\S 2.1]{Goertz-AlgGeom}.
    A locally ringed space $(X, \mathcal{O}_X)$ is a topological space $X$ together with a sheaf of commutative rings $\mathcal{O}_X$ on $X$.
    If a locally ringed space $(X, \mathcal{O}_X)$ is isomorphic to some $\big( \Spec \left( R \right), \mathcal{O}_{\Spec \left( R \right)} \big)$, where $R$ is a commutative ring, then $(X, \mathcal{O}_X)$ is called an affine scheme.
    A scheme $(X, \mathcal{O}_X)$ is a locally ringed space which can be covered by open affine schemes $X = \bigcup_i U_i$, i.e.\ $\big( U_i = \Spec \left( R_i \right), \mathcal{O}_X \vert_{U_i} \big)$.
    The dimension of a scheme $(X, \mathcal{O}_X)$ is defined to be the dimension of its  underlying topological space $X$.
    For a scheme $(X, \mathcal{O}_X)$, if the sheaf $\mathcal{O}_X$ is clear from context we just write $X$ for the scheme.
    For a general introduction into the scheme-theoretic formulation of algebraic geometry we refer the reader to \cite{Goertz-AlgGeom}.

    For a ring $R$ the $n$-dimensional affine space $\AffSp{n}{R}$ is defined to be the scheme $\Spec \big( \allowbreak R [x_1, \dots, x_n] \big)$, and for an ideal $I \subset R [x_1, \dots, x_n]$ the scheme $\mathcal{V} (I) = \Spec \big( R [x_1, \allowbreak \dots, x_n] / I \big)$ is called (affine) vanishing scheme of $I$.
    As scheme, the $n$-dimensional projective space $\ProjSp{n}{R}$ over $R$ is obtained by gluing $n + 1$ copies of the affine space $\AffSp{n}{R}$, see \cite[\S 3.6]{Goertz-AlgGeom}.
    Let $R [X_0, \dots, X_n]$ be graded via the degree, and let $I \subset R [X_0, \dots, X_n]$ be a homogeneous ideal.
    Analogously to the construction of $\ProjSp{n}{R}$, one can construct a scheme $\mathcal{V}_+ (I)$ together with a closed immersion $\iota: \mathcal{V}_+ (I) \to \ProjSp{n}{R}$.
    Therefore, $\mathcal{V}_+ (I)$ is called the projective vanishing scheme of $I$, see \cite[\S 3.7]{Goertz-AlgGeom}.

    Let $k$ be a field, let $k \subset K$ be a field extension, and let $X$ be a $k$-scheme.
    The $K$-valued points of $X$ are given by $X (K) = \Hom_k \big( \Spec (K), X \big)$.
    For affine/projective vanishing schemes the $K$-valued points have an elementary interpretation.
    \begin{ex}\label[ex]{Ex: k-valued points}
        Let $k$ be a field, let $k \subset K$ be a field extension, let $I \subset k [x_1, \dots, x_n]$ be an ideal, and let $J \subset k [X_0, \dots, X_n]$ be a homogeneous ideal.
        \begin{enumerate}
            \item $\AffSp{n}{k} (k) = k^n$ \cite[Example~4.2]{Goertz-AlgGeom}.
            I.e., the $k$-valued points correspond to the elementary definition of the affine space.

            \item $\mathcal{V} (I) \big( k \big) = \big\{ \mathbf{x} \in \AffSp{n}{k} (k) \; \big\vert \; \forall f \in I \!: f (\mathbf{x}) = 0 \big\}$ \cite[Example 4.3]{Goertz-AlgGeom}.
            I.e., the $k$-valued points correspond to the elementary definition of the affine variety.

            \item $\ProjSp{n}{k} (k)$ is given by points $\mathbf{x} \in \big( k^{n + 1} \setminus \{ \mathbf{0} \} \big) / \sim$ with the equivalence relation $\mathbf{x} \sim \mathbf{y} \Leftrightarrow \exists a \in k^\times \!: \mathbf{x} = a \cdot \mathbf{y}$ \cite[Exercise 4.6]{Goertz-AlgGeom}.
            I.e., the $k$-valued points correspond to the elementary definition of the projective space.

            \item $\mathcal{V}_+ (I) \big( K \big) = \big\{ \mathbf{x} \in \ProjSp{n}{K} (K) \; \big\vert \; \forall f \in J \!: f (\mathbf{x}) = 0 \big\}$ \cite[Example 5.3]{Goertz-AlgGeom}.
            I.e., the $K$-valued points correspond to the elementary definition of the projective variety.
        \end{enumerate}
    \end{ex}

    Let $k$ be a field, and let $F, G \in k [X_0, X_1, X_2]$.
    The vanishing scheme $\mathcal{V}_+ (F) \subset \ProjSp{n}{k}$ is also called \emph{plane curve}.
    It is well-known that the scheme-theoretic intersection of two plane curves \cite[Example~4.38]{Goertz-AlgGeom} is given by
    \begin{equation}
        \mathcal{V}_+ (F) \cap \mathcal{V}_+ (G) = \mathcal{V}_+ (F, G) \subset \ProjSp{2}{k}.
    \end{equation}

    Counting the number of points in intersections of vanishing schemes is a standard problem in algebraic geometry.
    Next, we recall the intersection number tailored for plane curves.
    \begin{defn}[{\cite[Definition~5.60]{Goertz-AlgGeom}}]
        Let $k$ be a field, and let $C, D \subset \ProjSp{2}{k}$ be two plane curves such that $Z := C \cap D$ is a $k$-scheme of dimension $0$.
        Then we call $i (C, D) := \dim_k \big( \Gamma (Z, \mathcal{O}_Z) \big)$ the intersection number of $C$ and $D$.
        For $z \in Z$ we call $i_z (C, D) := \dim_k (\mathcal{O}_{Z, z})$ the intersection number of $C$ and $D$ at $z$.
    \end{defn}

    From the definition it is easily verified that
    \begin{equation}
        i (C, D) = \sum_{z \in C \cap D} i_z (C, D).
    \end{equation}

    B\'ezout's famous theorem states that the intersection number of two plane curves together with their multiplicities is equal to the product of their degrees.
    \begin{thm}[{B\'ezout's Theorem, \cite[Theorem~5.61]{Goertz-AlgGeom}}]\label[thm]{Th: Bezout}
        Let $k$ be a field, and let $C = \mathcal{V}_+ (F)$ and $D = \mathcal{V}_+ (G)$ be plane curves in $\ProjSp{2}{k}$ given by polynomials without a common factor.
        Then
        \begin{equation*}
            i (C, D) = \degree{F} \cdot \degree{G}.
        \end{equation*}
        In particular, the intersection $C \cap D$ is non-empty and consists of a finite number of closed points.
    \end{thm}

    \subsection{Polytope Method}
    It is well-known that polynomial rings over fields $P = k [x_1, \dots, x_n]$ are unique factorization domains.
    A polynomial $f \in P$ is called \emph{irreducible} if for every factorization $f = g \cdot h$, with $g, h \in P$, one hast that either $g \in k^\times$ or $h \in k^\times$.
    Additionally, an irreducible polynomial $f \in P$ is called \emph{absolutely irreducible} if it remains irreducible over every algebraic field extension $k \subset K$.

    A standard algorithmic problem in commutative algebra is the decomposition of $f$ into its irreducible factors with multiplicities.
    However, it is often useful to decide whether $f$ is irreducible or not without explicit factorization of $f$.
    For the latter purpose, mathematicians developed various irreducibility criteria.
    In this paper, we work with Gao's polytope method \cite{Gao-Irreducibility}.

    Let $S \subset \R^n$ be a subset, then $S$ is called \emph{convex} if for any two points $\mathbf{x}, \mathbf{y} \in S$ the line segment for $\mathbf{x}$ to $\mathbf{y}$ is also contained in $S$, i.e.\
    \begin{equation}
        \forall 0 \leq t \leq 1 \!: \mathbf{x} + t \cdot (\mathbf{y} - \mathbf{x}) = (1 - t) \cdot \mathbf{x} + t \cdot \mathbf{y} \in S.
    \end{equation}
    The smallest convex set of $\R^n$ which contains $S$ is called the \emph{convex closure} $\conv \left( S \right)$ of $S$.
    The following identity is well-known
    \begin{equation}\label{Equ: convex closure}
        \conv \left( S \right) = \left\{ \sum_{i = 1}^{n} t_i \cdot \mathbf{a}_i \; \middle\vert \; \mathbf{a}_i \in S,\ t_i \geq 0,\ \sum_{i = 1}^{n} t_i = 1 \right\}.
    \end{equation}
    For a finite set $S = \{ \mathbf{a}_1, \dots, \mathbf{a}_n \}$ we also denote the convex closure as $\braket{\mathbf{a}_1, \dots, \mathbf{a}_n}$.
    The convex closure of a finite set $S$ is also called \emph{polytope}, and a point of a polytope is called \emph{vertex} if it is not on the line segment of any other two points of the polytope.
    With \Cref{Equ: convex closure} it is easy to verify that a polytope is equal to the convex closure of its vertices.

    Returning to polynomials, any $f \in k [x_1, \dots, x_n]$ can be expressed as finite sum $f = \sum_{i} c_i \cdot x_1^{d_{i_1}} \cdots x_n^{d_{i_n}}$, where $c_i \in k^\times$.
    Naturally, we can consider the exponent vectors $\mathbf{d}_i = \left( d_{i_1}, \dots, d_{i_n} \right)^\intercal \subset \Z_{\geq 0}^n$ as elements of $\R^n$.
    The \emph{Newton polytope} $\mathcal{N}_f$ of $f$ is then defined to be the convex closure $\conv \big( \{ \mathbf{d}_i \}_i \big)$ of $f$'s exponent vectors.

    For two sets $A, B \subset \R^n$, their \emph{Minkowski sum} is defined as $A + B = \{ a + b \mid a \in A,\ \allowbreak b \in B \}$.
    Let $f, g, h \in k [x_1, \dots, x_n]$ be such that $f = g \cdot h$, fundamental to Gao's irreducibility criterion is the identity \cite[Lemma~2.1]{Gao-Irreducibility}
    \begin{equation}
        \mathcal{N}_f = \mathcal{N}_g + \mathcal{N}_h.
    \end{equation}

    A point $\mathbf{x} \in \R^n$ is called \emph{integral} if $\mathbf{x} \in \Zn$, and a polytope is called integral if all its vertices are integral.
    An integral polytope $C$ is called \emph{integrally decomposable} if it can be written as Minkowski sum $C = A + B$ with $A$ and $B$ integral polytopes too.
    Otherwise, $C$ is called \emph{integrally indecomposable}.
    Finally, we can express Gao's irreducibility criterion.
    \begin{thm}[{Gao's Irreducibility Criterion, \cite[\S 2]{Gao-Irreducibility}}]\label[thm]{Th: absolutely irreducible}
        Let $k$ be a field, and let $f \in k [x_1, \dots, x_n]$ be a non-zero polynomial not divisible by any $x_i$.
        If the Newton polytope of $f$ is integrally indecomposable, then $f$ is absolutely irreducible.
    \end{thm}

    In this paper we will encounter triangles, i.e.\ polytopes generated by three distinct points $\braket{\mathbf{x}, \mathbf{y}, \mathbf{z}} \subset \R^n$.
    For integral triangles, Gao has fully characterized being integrally indecomposable.
    Let $\mathbf{a} = \left( a_1, \dots, a_n \right)^\intercal, \mathbf{b} \in \Zn$, for ease of writing we abbreviate $\gcd \left( \mathbf{a} \right) = \gcd \left( a_1, \dots, a_n \right)$, and analogously for $\gcd \left( \mathbf{a}, \mathbf{b} \right)$.
    \begin{cor}[{\cite[Corollary~4.5]{Gao-Irreducibility}}]\label[cor]{Cor: triangle indecomposable}
        Let $\mathbf{v}_0, \mathbf{v}_1, \mathbf{v}_2 \in \Zn$ be three points not all on one line.
        Then the triangle $\braket{\mathbf{v}_0, \mathbf{v}_1, \mathbf{v}_2}$ is integrally indecomposable if and only if
        \begin{equation*}
            \gcd \left( \mathbf{v}_0 - \mathbf{v}_1, \mathbf{v}_0 - \mathbf{v}_2 \right) = 1.
        \end{equation*}
    \end{cor}

    \subsection{Triangles}\label{Sec: Triangles}
    For our study of triangle Newton polytopes we also need some basic properties of triangles, which can be proven with elementary Euclidean geometry.
    For an introduction to Euclidean geometry we refer the reader to \cite{Hartshorn-Euclid}.

    By mirroring, any triangle $\Delta = \mathbf{A}\mathbf{B}\mathbf{C} \subset \R^2$ can be extended to a parallelogram $P = \mathbf{A}\mathbf{B}\mathbf{D}\mathbf{C} \subset \R^2$, see \Cref{Fig: mirrored triangle}.
    Let $\mathbf{a} = \mathbf{B} - \mathbf{A} =
    \begin{pmatrix}
        a_x \\ a_y
    \end{pmatrix}
    $ and $\mathbf{b} = \mathbf{C} - \mathbf{A} =
    \begin{pmatrix}
        b_x \\ b_y
    \end{pmatrix}
    $ denote the edge vectors spanning the parallelogram, then the area of $P$ is given by the determinant
    \begin{equation}
        \mathcal{A}_P = \abs{\det
        \begin{pmatrix}
            a_x & b_x \\
            a_y & b_y
        \end{pmatrix}
        }
        \qquad \Longrightarrow \qquad \mathcal{A}_\Delta = \frac{1}{2} \cdot
        \abs{ \det
        \begin{pmatrix}
            a_x & b_x \\
            a_y & b_y
        \end{pmatrix}
        }
        .
    \end{equation}

    \clearpage 
    \begin{figure}[H]
    	\centering
    	\begin{tikzpicture}
    		\coordinate[label=below:$\mathbf{A}$] (A) at (0, 0);
    		\coordinate[label=below:$\mathbf{B}$] (B) at (3, 1);
    		\coordinate[label=above:$\mathbf{C}$] (C) at (2, 4);
    		\coordinate[label=above:$\mathbf{D}$] (D) at (5, 5);

    		\draw[thick] (A) -- node[below] {$\mathbf{a}$} (B) -- (C) -- cycle;
    		\draw[thick, loosely dashed] (B) -- (C) -- (D) -- node[right] {$\mathbf{b}$} cycle;
    	\end{tikzpicture}
    	\caption{Mirrored triangle.}
        \label{Fig: mirrored triangle}
    \end{figure}
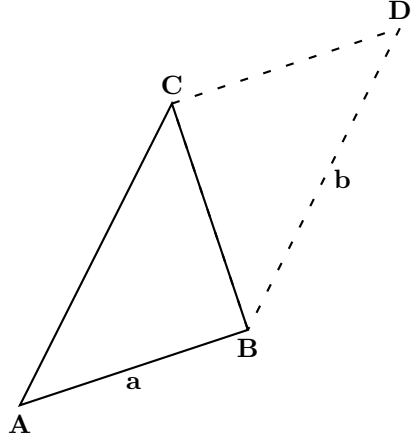

    Let $\Delta = \mathbf{A}\mathbf{B}\mathbf{C} \subset \R^2$ be a triangle, and let $\mathbf{P} \in \R^2$ be a point not lying on one of the edges $\mathbf{B} \mathbf{A}$, $\mathbf{C} \mathbf{B}$ or $\mathbf{A} \mathbf{C}$.
    To decide whether $\mathbf{P}$ lies within the triangle $\Delta$, we consider the three triangles $\Delta_1 = \mathbf{A}\mathbf{B}\mathbf{P}$, $\Delta_2 = \mathbf{A}\mathbf{P}\mathbf{C}$ and $\Delta_3 = \mathbf{P}\mathbf{B}\mathbf{C}$.
    If the area identity $\mathcal{A}_{\Delta} = \mathcal{A}_{\Delta_1} + \mathcal{A}_{\Delta_2} + \mathcal{A}_{\Delta_3}$ is satisfied, then $\mathbf{P}$ indeed lies within $\Delta$, otherwise the identity is not satisfied.
    This property is illustrated in \Cref{Fig: triangles}.
    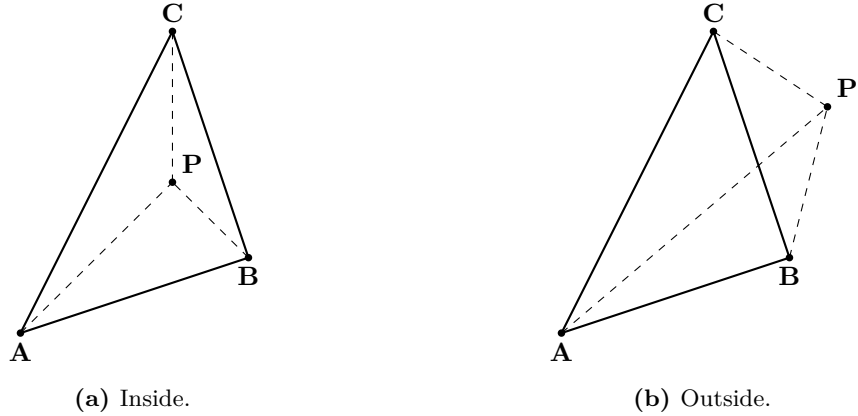
\begin{figure}[H]
        \centering
        \begin{subfigure}[H]{0.49\linewidth}
            \centering
            \begin{tikzpicture}
                \coordinate[label=below:$\mathbf{A}$] (A) at (0, 0);
                \coordinate[label=below:$\mathbf{B}$] (B) at (3, 1);
                \coordinate[label=above:$\mathbf{C}$] (C) at (2, 4);
                \coordinate[label=above right:$\mathbf{P}$] (P) at (2, 2);

                \draw[thick] (A) -- (B) -- (C) -- cycle;
                \draw[dashed] (A) -- (P);
                \draw[dashed] (B) -- (P);
                \draw[dashed] (C) -- (P);

                \draw[fill=black, thick] (A) circle[radius=1pt, thick];
                \draw[fill=black, thick] (B) circle[radius=1pt, thick];
                \draw[fill=black, thick] (C) circle[radius=1pt, thick];
                \draw[fill=black, thick] (P) circle[radius=1pt, thick];
            \end{tikzpicture}
            \caption{Inside.}
        \end{subfigure}
        \hfill
        \begin{subfigure}[H]{0.49\linewidth}
            \centering
            \begin{tikzpicture}
                \coordinate[label=below:$\mathbf{A}$] (A) at (0, 0);
                \coordinate[label=below:$\mathbf{B}$] (B) at (3, 1);
                \coordinate[label=above:$\mathbf{C}$] (C) at (2, 4);
                \coordinate[label=above right:$\mathbf{P}$] (P) at (3.5, 3);

                \draw[thick] (A) -- (B) -- (C) -- cycle;
                \draw[dashed] (A) -- (P);
                \draw[dashed] (B) -- (P);
                \draw[dashed] (C) -- (P);

                \draw[fill=black, thick] (A) circle[radius=1pt, thick];
                \draw[fill=black, thick] (B) circle[radius=1pt, thick];
                \draw[fill=black, thick] (C) circle[radius=1pt, thick];
                \draw[fill=black, thick] (P) circle[radius=1pt, thick];
            \end{tikzpicture}
            \caption{Outside.}
        \end{subfigure}
        \caption{Illustration of a point lying inside or outside a triangle.}
        \label{Fig: triangles}
    \end{figure}

    \subsection{Gr\"obner Bases}
    Gr\"obner bases are a tool to solve computational problems of ideals $I \subset k [x_1, \dots, x_n]$ like the membership problem, computation of the dimension or computation of the set of zeros.
    For the definition of Gr\"obner bases we first have to introduce term orders on $k [x_1, \dots, x_n]$.
    \begin{defn}[{Term Order, \cite[Chapter 2 \S 2 Definition 1]{Cox-Ideals}}]\label[defn]{Def: term order}
        Let $k$ be a field, a term order $>$ on $k [x_{1}, \dots, x_{n}] = k [\mathbf{x}]$ is a relation $>$ on $\Zn_{\geq 0}$, or equivalently, a relation on the set of monomials $\mathbf{x}^{\mathbf{a}} = \prod_{i = 1}^{n} x_{i}^{a_{i}}$, $\mathbf{a} \in \Zn_{\geq 0}$, satisfying:
        \begin{enumerate}[label=(\roman*)]
            \item $>$ is a total ordering on $\Zn_{\geq 0}$.

            \item If $\mathbf{a} > \mathbf{b}$ and $\mathbf{c} \in \Zn_{\geq 0}$, then $\mathbf{a} + \mathbf{c} > \mathbf{b} + \mathbf{c}$.

            \item $>$ is a well-ordering on $\Zn_{\geq 0}$.
            I.e., every non-empty subset of $\Zn_{\geq 0}$ has a smallest element under $>$.
        \end{enumerate}
    \end{defn}

    In this paper we work with the \emph{lexicographic} (LEX) and the \emph{degree reverse lexicographic} (DRL) term order:
    \begin{itemize}
        \item $\mathbf{a} >_{LEX} \mathbf{b}$ if the leftmost non-zero entry of $\mathbf{a} - \mathbf{b}$ is positive.

        \item $\mathbf{a} >_{DRL} \mathbf{b}$ if either $\sum_{i = 1}^{n} a_{i} > \sum_{i = 1}^{n} b_{i}$, or $\sum_{i = 1}^{n} a_{i} = \sum_{i = 1}^{n} b_{i}$ and the rightmost non-zero entry of $\mathbf{a} - \mathbf{b}$ is negative.
    \end{itemize}

    Let $f \in P = k [x_1, \dots, x_n]$ be a non-zero polynomial, and let $>$ be a term order on $P$.
    The largest monomial present in $f$ is called the \emph{$>$-leading monomial}, and the coefficient of the largest monomial is called \emph{$>$-leading coefficient}.
    We denote them as $\LM_> \left( f \right)$ and $\LC_> \left( f \right)$ respectively.
    Now we can recall the definition of a Gr\"obner basis.
    \begin{defn}[{Gr\"obner Basis, \cite[Chapter 2 \S 5 Definition 5]{Cox-Ideals}}]\label[defn]{Def: Groebner basis}
        Let $k$ be a field, let $I \subset P = k [x_{1}, \allowbreak \dots, x_{n}]$ be a non-zero ideal, and let $>$ be a term order on $P$.
        A finite subset $\mathcal{G} \subset I$ is said to be a $>$-Gr\"obner basis if
        \begin{equation*}
            \big( \LM_> (g) \; \big\vert \; g \in \mathcal{G} \big) = \big( \LM_> (f) \; \big\vert \; f \in I \big).
        \end{equation*}
    \end{defn}

    Next we describe a computational criterion to detect Gr\"obner bases due to Buchberger.
    Let $f, g \in P = k [x_1, \dots, x_n]$ and let $>$ be a term order on $P$.
    The \emph{S-polynomial} of $f$ and $g$ is defined as
    \begin{equation}
        S_> (f, g) = \frac{\lcm \big( \LM_> (f), \LM_> (g) \big)}{\LC_> (f)} \cdot f - \frac{\lcm \big( \LM_> (f), \LM_> (g) \big)}{\LC_> (g)} \cdot g.
    \end{equation}
    Buchberger's criterion requires the multivariate polynomial division algorithm, readers can find its description in \cite[Chapter~2~\S 3]{Cox-Ideals}.
    \begin{thm}[{Buchberger's Criterion, \cite[Chapter~2~\S 6~Theorem~6]{Cox-Ideals}}]\label[thm]{Th: Buchberger criterion}
        Let $k$ be a field, let $I \subset P = k [x_1 \dots, x_n]$, and let $>$ be a term order on $P$.
        Then a basis $\mathcal{G} = \{ g_1, \dots, g_m \} \subset I$ of $I$ is a $>$-Gr\"obner basis if and only if for all pairs $i \neq j$, the remainder on division of $S_> (g_i, g_j)$ by $\mathcal{G}$ (listed in some order) is zero.
    \end{thm}

    This criterion yields a very simple algorithm to compute Gr\"obner basis.
    For the input set $\mathcal{F}$, compute the remainders of all possible S-polynomials with respect to $\mathcal{F}$.
    If a remainder is non-zero, add it to $\mathcal{F}$ and repeat until Buchberger's criterion is satisfied.
    This procedure is known as Buchberger's algorithm \cite{Buchberger}.

    Let $>$ be a term order on $P = k [x_1, \dots, x_n]$.
    An ideal $I \subset P = k [x_1, \dots, x_n]$ is zero-dimensional, i.e.\ $\dim \left( P / I \right) = 0$ if and only if for every $1 \leq i \leq n$ there exists $f \in I$ and $d_i \in \Z_{\geq 1}$ such that $\LM_> \left( f \right) = x_i^{d_i}$, see \cite[Chapter~5~\S 3~Theorem~6]{Cox-Ideals} and \cite[Theorem 5.11]{Kemper-CommutativeAlgebra}.
    In particular, let $d = \dim_k \left( P / I \right) < \infty$ denote the $k$-vector space dimension of the quotient ring, then \cite[Chapter~5~\S 3~Proposition~7]{Cox-Ideals}
    \begin{equation}
        \abs{\mathcal{V} (I) \big( k \big)} = \abs{\big\{ \mathbf{x} \in k^n \; \big\vert \; \forall f \in I \!: f (\mathbf{x}) = 0 \big\}} \leq d,
    \end{equation}
    with equality if $k$ is algebraically closed and $I$ is radical.

    For a general introduction into the theory of Gr\"obner bases we refer the reader to \cite{Cox-Ideals,Kreuzer-CompAlg1,Kreuzer-CompAlg2}.

    \section{Irreducibility of \texorpdfstring{$1$}{1}-Boomerang Polynomials}\label{Sec: Irreducibility of Boomerang Polynomials}
    We now prove our first main result of this paper, the absolute irreducibility of polynomials $f (x) - f (y) + a$, which yields the $1$-Boomerang uniformity bound by B\'ezout's theorem.

    As preparation, we recall that irreducibility of a polynomial is invariant under linear variable transformations.
	\begin{lem}\label[lem]{Lem: change of variables irreducible}
		Let $k$ be a field, let $f \in k [\mathbf{x}] = k [x_1, \dots, x_n]$, and let $\mathbf{M} \in \Fqnxn$ be such that $\det \left( \mathbf{M} \right) \neq 0$.
		Then $f (\mathbf{x})$ is irreducible if and only if $f (\mathbf{M} \mathbf{x})$ is irreducible
	\end{lem}
	\begin{proof}
        Let $\mathbf{y} = \mathbf{M} \mathbf{x}$, then we have a ring isomorphism $\phi: k [\mathbf{x}] \to k [\mathbf{y}]$.
        Now assume that $\phi (f) = \phi (f_1) \cdot \phi (f_2)$, since $\phi$ is an isomorphism we have $f = f_1 \cdot f_2$.
        If $f$ is irreducible, then either $f_1 \in k^\times$ or $f_2 \in k^\times$.
        Since $\phi (k^\times) = k^\times$ this implies that $\phi (f)$ is irreducible too.
        The other direction follows identical.
	\end{proof}

    The irreducibility proof is an application of the polytope method.
	\begin{prop}\label[prop]{Prop: irreducible polynomial}
        Let $k$ be a field of characteristic $p \in \Z_{\geq 0}$.
		\begin{enumerate}
			\item Let $d \in \Z_{\geq 2}$, and let $i_1, i_2 \in \Z_{\geq 1}$.
			Then the polytope
			\begin{equation*}
				\Bigg<
				\begin{pmatrix}
					0 \\ 0
				\end{pmatrix}
				,
				\begin{pmatrix}
					1 \\ d - 1
				\end{pmatrix}
				,
				\begin{pmatrix}
					d \\ 0
				\end{pmatrix}
				,
				\begin{pmatrix}
					i_1 \\ i_2
				\end{pmatrix}
				\Bigg>_{2 \leq i_1 + i_2 \leq d}
				=
				\Bigg<
				\begin{pmatrix}
					0 \\ 0
				\end{pmatrix}
				,
				\begin{pmatrix}
					1 \\ d - 1
				\end{pmatrix}
				,
				\begin{pmatrix}
					d \\ 0
				\end{pmatrix}
				\Bigg> \subset \R^2
			\end{equation*}
			is an integrally indecomposable triangle.

			\item Let $f \in k [X]$ with $d = \degree{f} > 0$, and let $a \in k^\times$.
			If $p > 0$, then assume in addition that $\gcd \left( d, p \right) = 1$.
			Then
			\begin{equation*}
				F = f (x) - f (y) + a \in k [x, y]
			\end{equation*}
			is absolutely irreducible over $k$.
		\end{enumerate}
	\end{prop}
	\begin{proof}
		We first consider the polynomial $F$.
		Let $f = \sum_{i = 1}^{d} a_i \cdot x^i$,\footnote{
            Since the constant term of $f$ is canceled in $F$ we can without loss of generality assume that $f (0) = 0$.
            }
        then
		\begin{align*}
			F (x + y, y) &= f (x + y) - f (y) + a \\
			&= \left( \sum_{i = 1}^{d} a_i \cdot (x + y)^i \right) - f (y) + a \\
			&= \left( f (x) + \sum_{i = 2}^{d} a_i \cdot \sum_{j = 1}^{i - 1} \binom{i}{j} \cdot x^j \cdot y^{i - j} + f (y) \right) - f(y) + a \\
			&= f (x) + \sum_{i = 2}^{d} a_i \cdot \sum_{j = 1}^{i - 1} \binom{i}{j} \cdot x^j \cdot y^{i - j} + a.
		\end{align*}
        In particular, in characteristic $p > 0$ we have by assumption for $i = d$ and $j = 1$ that $\binom{i}{j} = \binom{d}{1} = d \not\equiv 0 \mod p$, i.e.\ the monomial $x \cdot y^{d - 1}$ is present in $F (x + y, y)$.
		Therefore, all exponents which can be present in $F (x + y, y)$ are contained in the set
		\begin{equation*}
			\left\{
			\begin{pmatrix}
				0 \\ 0
			\end{pmatrix}
			,
			\begin{pmatrix}
				j \\ 0
			\end{pmatrix}
			,
			\begin{pmatrix}
				i_1 \\ i_2
			\end{pmatrix}
			\right\}_{\subalign{1 \leq &j \leq d, \\ 2 \leq &i_1 + i_2 \leq d}}.
		\end{equation*}
		Recall that for a convex set $\mathcal{C} \subset \R^n$ and a point $\mathbf{p} \in \mathcal{C}$ we have for the convex closure that $\braket{\mathcal{C}, \mathbf{p}} = \mathcal{C}$.
		Thus,
		\begin{equation}\label{Equ: asserted triangle}
			\Bigg<
			\begin{pmatrix}
				0 \\ 0
			\end{pmatrix}
			,
			\begin{pmatrix}
				j \\ 0
			\end{pmatrix}
			,
			\begin{pmatrix}
				i_1 \\ i_2
			\end{pmatrix}
			\Bigg>_{\subalign{1 \leq &j \leq d, \\ 2 \leq &i_1 + i_2 \leq d}}
			= \Bigg<
			\begin{pmatrix}
				0 \\ 0
			\end{pmatrix}
			,
			\begin{pmatrix}
				1 \\ d - 1
			\end{pmatrix}
			,
			\begin{pmatrix}
				d \\ 0
			\end{pmatrix}
			,
			\begin{pmatrix}
				i_1 \\ i_2
			\end{pmatrix}
			\Bigg>_{2 \leq i_1 + i_2 \leq d}.
		\end{equation}
        (The points $(j, 0)^\intercal$ trivially lie on the line from $(0, 0)^\intercal$ to $(d, 0)^\intercal$.)
		Now we want to prove that this polytope is indeed a triangle.
		In \Cref{Fig: triangle illustration} we provide an illustration of the polytope from the assertion.
		\begin{figure}[H]
			\centering
			\begin{tikzpicture}
				\coordinate[label=below:${(0, 0)}$] (A) at (0, 0);
				\coordinate[label=above:${(1, d - 1)}$] (B) at (1, 7 - 1);
				\coordinate[label=below:${(d, 0)}$] (C) at (7, 0);

				\draw[thick] (A) -- (B) -- (C) -- cycle;

				\draw[fill=black, thick] (A) circle[radius=1pt, thick];
				\draw[fill=black, thick] (B) circle[radius=1pt, thick];
				\draw[fill=black, thick] (C) circle[radius=1pt, thick];

				\draw[fill=black, thick] (1, 5) circle[radius=1pt, thick];
				\draw[fill=black, thick] (2, 5) circle[radius=1pt, thick];

				\draw[fill=black, thick] (1, 4) circle[radius=1pt, thick];
				\draw[fill=black, thick] (2, 4) circle[radius=1pt, thick];
				\draw[fill=black, thick] (3, 4) circle[radius=1pt, thick];

				\draw[thick, loosely dotted] (1.5, 3.5) -- (1.5, 2.5);
				\draw[thick, loosely dotted] (2.5, 3.5) -- (2.5, 2.5);

				\draw[black, thick] (1, 2) circle[radius=1pt, thick];
				\draw[fill=black, thick] (2, 2) circle[radius=1pt, thick];
				\draw[fill=black, thick] (3, 2) circle[radius=1pt, thick];
				\draw[fill=black, thick] (4, 2) circle[radius=1pt, thick];
				\draw[fill=black, thick] (5, 2) circle[radius=1pt, thick];

				\draw[fill=black, thick] (1, 1) circle[radius=1pt, thick];
				\draw[fill=black, thick] (2, 1) circle[radius=1pt, thick];
				\draw[fill=black, thick] (3, 1) circle[radius=1pt, thick];
				\draw[fill=black, thick] (4, 1) circle[radius=1pt, thick];
				\draw[fill=black, thick] (5, 1) circle[radius=1pt, thick];
				\draw[fill=black, thick] (6, 1) circle[radius=1pt, thick];
			\end{tikzpicture}
			\caption{Illustration of the triangle.}
			\label{Fig: triangle illustration}
		\end{figure}
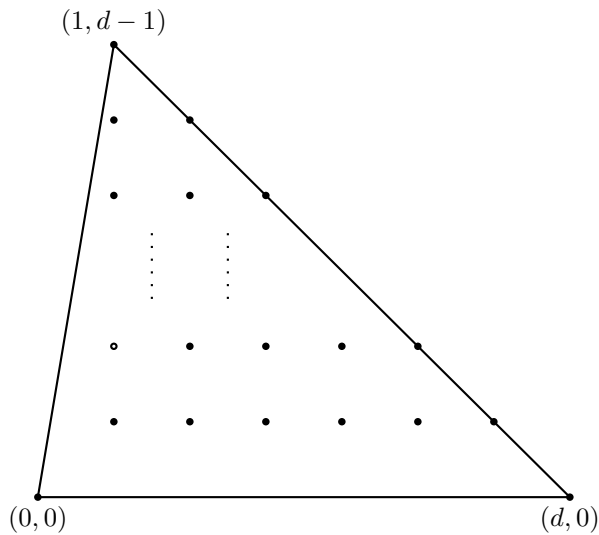
		For ease of notation, let $\Delta = \Bigg<
		\begin{pmatrix}
			0 \\ 0
		\end{pmatrix}
		,
		\begin{pmatrix}
			1 \\ d - 1
		\end{pmatrix}
		,
		\begin{pmatrix}
			d \\ 0
		\end{pmatrix}
		\Bigg>
		$.
		It is clear that every point $\mathbf{p} = \left( i_1, i_2 \right)^\intercal$ with $i_1 + i_2 = d$ lies on the line from $\left( 1, d - 1 \right)^\intercal$ to $\left( d, 0 \right)^\intercal$, so they are indeed in the triangle.
		For the other points $\mathbf{p} = \left( i_1, i_2 \right)^\intercal$ with $2 \leq i_1 + i_2 < d$, we consider the triangles
		\begin{equation*}
			\Delta_1 =
			\Bigg<
			\mathbf{p},
			\begin{pmatrix}
				1 \\ d - 1
			\end{pmatrix}
			,
			\begin{pmatrix}
				d \\ 0
			\end{pmatrix}
			\Bigg>, \qquad
			\Delta_2 =
			\Bigg<
			\begin{pmatrix}
				0 \\ 0
			\end{pmatrix}
			,
			\mathbf{p}
			,
			\begin{pmatrix}
				d \\ 0
			\end{pmatrix}
			\Bigg>, \qquad
			\Delta_3 =
			\Bigg<
			\begin{pmatrix}
				0 \\ 0
			\end{pmatrix}
			,
			\begin{pmatrix}
				1 \\ d - 1
			\end{pmatrix}
			,
			\mathbf{p}
			\Bigg>,
		\end{equation*}
		and verify whether we have the area identity
		\begin{equation*}
			\mathcal{A}_{\Delta} = \mathcal{A}_{\Delta_1} + \mathcal{A}_{\Delta_2} + \mathcal{A}_{\Delta_3},
		\end{equation*}
		or not, see \Cref{Sec: Triangles} and \Cref{Fig: triangles}.
		Let $k = i_1 + i_2$, then we have the triangle areas, see \Cref{Sec: Triangles},
		\begin{alignat*}{2}
			\mathcal{A}_{\Delta} &= \frac{1}{2} \cdot \abs{\det
				\begin{pmatrix}
					d & 0 \\
					1 & d - 1
				\end{pmatrix}
			}
			&&= \frac{d \cdot (d - 1)}{2}, \\
			\mathcal{A}_{\Delta_1} &= \frac{1}{2} \cdot \abs{\det
				\begin{pmatrix}
					d - i_1 & -i_2 \\
					1 - i_1 & d - 1 - i_2
				\end{pmatrix}
			}
			&&= \frac{\abs{d \cdot (d - 1) - k \cdot (d - 1)}}{2}, \\
			\mathcal{A}_{\Delta_2} &= \frac{1}{2} \cdot \abs{\det
				\begin{pmatrix}
					d & 0 \\
					i_1 & i_2
				\end{pmatrix}
			}
			&&= \frac{i_2 \cdot d}{2}, \\
			\mathcal{A}_{\Delta_3} &= \frac{1}{2} \cdot \abs{\det
				\begin{pmatrix}
					i_1 & i_2 \\
					1 & d - 1
				\end{pmatrix}
			}
			&&= \frac{\abs{i_1 \cdot d - k}}{2}.
		\end{alignat*}
		Since $i_1 \geq 1$ and $k \leq d$ we can drop the absolute values in $\mathcal{A}_{\Delta_1}$ and $\mathcal{A}_{\Delta_3}$.
		Now we observe that
		\begin{equation*}
			\frac{d \cdot (d - 1) - k \cdot (d - 1)}{2} + \frac{i_2 \cdot d}{2} +  \frac{i_1 \cdot d - k}{2} = \frac{d \cdot (d - 1)}{2}.
		\end{equation*}
		So, all the points from the assertion indeed are inside the triangle $\Delta$ or its boundary, hence the convex closure from \Cref{Equ: asserted triangle} is indeed a triangle.

		For integral indecomposability of $\Delta$, note that
		\begin{equation*}
			\gcd \Bigg(
			\begin{pmatrix}
				1 \\ d - 1
			\end{pmatrix}
			-
			\begin{pmatrix}
				0 \\ 0
			\end{pmatrix}
			,
			\begin{pmatrix}
				d \\ 0
			\end{pmatrix}
			-
			\begin{pmatrix}
				0 \\ 0
			\end{pmatrix}
			\Bigg)
			= \gcd \left( 0, 1, d - 1, d \right) = 1,
		\end{equation*}
		and the claim follows from \Cref{Cor: triangle indecomposable}.
		This proves (1).

		Returning to $F (x + y, y)$, due to the assumption that $\gcd \left( d, p \right) = 1$ if $p > 0$ the monomials $1, x^d, x \cdot y^{d - 1}$ have non-zero coefficients.
		Combining all our previous observation we conclude that $\Delta$ is exactly the Newton polytope of $F (x + y, y)$.
		So, absolute irreducibility of $F (x + y, y)$ follows from \Cref{Th: absolutely irreducible}, and absolute irreducibility of $F (x, y)$ follows from \Cref{Lem: change of variables irreducible}.
		This proves (2).
	\end{proof}

    Now we apply B\'ezout's theorem to the $1$-Boomerang polynomials.
    \begin{thm}\label[thm]{Th: degree bound c = 1}
    	Let $\Fq$ be a finite field, let $f \in \Fq [x]$ with $d = \degree{f} > 0$ be a polynomial such that $\gcd \left( d, q \right) = 1$, and let $a, b \in \Fqx$.
    	Then ${}_1 \mathcal{B}_f (a, b) \leq d \cdot (d - 2)$.
    \end{thm}
    \begin{proof}
    	Let $a, b \in \Fqx$.
        We consider the $1$-Boomerang equations with the substitution $z = x + y$
    	\begin{align*}
    		F_1 &= f (z) - f (x) - b, \\
    		F_2 &= f (z + a) - f (x + a) - b,
    	\end{align*}
        as polynomials in $\Fq [x, z]$.
    	By \Cref{Prop: irreducible polynomial} $F_1$ is absolutely irreducible.
    	Now let us consider
    	\begin{equation*}
    		F_2 - F_1 = \big( f (z + a) - f (z) \big) - \big( f (x + a) - f (x) \big).
    	\end{equation*}
    	Then, all points $z = x$ are roots of this polynomial.
    	In particular, via elementary methods, see e.g.\ \cite[Lemma~3.5]{Steiner-Boomerang} we can prove that $F_2 - F_1 = (z - x) \cdot G_2$, where $G_2 \in \Fq [x, z]$.
    	Let us substitute $z = x$ into $F_1$ which yields $b = 0$.
        Since $b \in \Fqx$ the boomerang equations do not have a symmetric solution.
    	Therefore, we can replace $F_2 - F_1$ by $G_2 = \frac{F_2 - F_1}{z - x}$.
    	Note that $\degree{G_2} \leq d - 2$.

    	Now we homogenize $F_1$ and $G_2$ with respect to a new variable $X_0$ and consider $F_1^\homog, G_2^\homog \in \Fq [X_0, X, Z]$ as plane curves in $\ProjSp{2}{\overline{\Fq}}$.
    	Since factorization is invariant under homogenization \cite[Proposition~4.3.2]{Kreuzer-CompAlg2} $F_1^\homog$ is still irreducible.
    	Since $\degree{G_2^\homog} = \degree{G_2} < \degree{F_1} = \degree{F_1^\homog}$ these curves cannot have a common irreducible component.
    	So, by B\'ezout's theorem (\Cref{Th: Bezout}) we have
    	\begin{align*}
    		i \Big( \mathcal{V}_+ \big( F_1^\homog \big), \mathcal{V}_+ \big( G_2^\homog \big) \Big)
    		&= \degree{F_1} \cdot \degree{G_2} \\
    		&= d \cdot (d - 2) \\
    		&= \sum_{z \in \mathcal{V}_+ \left( F_1^\homog \right) \cap \mathcal{V}_+ \left( G_2^\homog \right)} \dim_k \left( \mathcal{O}_{\mathcal{V}_+ \left( F_1^\homog \right) \cap \mathcal{V}_+ \left( G_2^\homog \right), z} \right)\\
    		&\geq \sum_{z \in \mathcal{V}_+ \left( F_1^\homog \right) \cap \mathcal{V}_+ \left( G_2^\homog \right)} 1,
    	\end{align*}
    	where the inequality follows from $\overline{\Fq} \hookrightarrow \mathcal{O}_{\mathcal{V}_+ \left( F_1^\homog \right) \cap \mathcal{V}_+ \left( G_2^\homog \right), z}$.
    	It follows from zero-dimensionality \cite[Lemmma~5.59]{Goertz-AlgGeom} that all points in $\mathcal{V}_+ \left( F_1^\homog \right) \cap \mathcal{V}_+ \left( G_2^\homog \right)$ \cite[Lemmma~5.59]{Goertz-AlgGeom} are closed.
    	Now recall the $\overline{\Fq}$-valued points of the intersection, see \Cref{Ex: k-valued points},
    	\begin{align*}
    		\Big( \mathcal{V}_+ \big( F_1^\homog \big) \cap \mathcal{V}_+ \big( G_2^\homog \big) \Big) \big( \overline{\Fq} \big)
    		&= \mathcal{V}_+ \left( F_1^\homog, G_2^\homog \right) \big( \overline{\Fq} \big) \\
    		&= \Big\{ \mathbf{x} \in \ProjSp{2}{\overline{\Fq}} \big( \overline{\Fq} \big) \; \Big\vert \; F_1^\homog (\mathbf{x}) = G_2^\homog (\mathbf{x}) = 0 \Big\}.
    	\end{align*}
    	Over an algebraically closed field $k$ the $k$-valued points are in one to one bijection with the closed points \cite[Corollary~3.36]{Goertz-AlgGeom}.
    	Therefore,
    	\begin{equation*}
    		\abs{\Big\{ \mathbf{x} \in \ProjSp{2}{\overline{\Fq}} \big( \overline{\Fq} \big) \; \Big\vert \; F_1^\homog (\mathbf{x}) = G_2^\homog (\mathbf{x}) = 0 \Big\}} \leq d \cdot (d - 2).
    	\end{equation*}
    	Finally, let $U_0 = \Big\{ (X_0, X, Z) \in \ProjSp{2}{\overline{\Fq}} \; \Big\vert \; X_0 \neq 0 \Big\}$, then we have \cite[Chapter~8~\S 2~Proposition~6]{Cox-Ideals}
    	\begin{equation*}
    		\Big\{ \mathbf{x} \in \ProjSp{2}{\overline{\Fq}} \big( \overline{\Fq} \big) \; \Big\vert \; F_1^\homog (\mathbf{x}) = G_2^\homog (\mathbf{x}) = 0 \Big\} \cap U_0
    		= \Big\{ \mathbf{x} \in \overline{\Fq}^2 \; \Big\vert \; F_1 (\mathbf{x}) = G_2 (\mathbf{x}) = 0 \Big\},
    	\end{equation*}
    	so also the number of affine solutions of the $1$-Boomerang polynomials is bounded by $d \cdot (d - 2)$.
    \end{proof}

    In characteristic $2$ the bound can be slightly improved.
    \begin{cor}\label[cor]{Cor: binary degree bound c = 1}
        Let $\Fq$ be a finite field of characteristic $2$, let $f \in \Fq [X]$ with $d = \degree{f} > 0$ be a polynomial such that $\gcd \left( d, q \right) = 1$, and let $a, b \in \Fqx$.
        Then ${}_1 \mathcal{B}_f (a, b) \leq d \cdot (d - 2) - 1$.
    \end{cor}
    \begin{proof}
        By \Cref{Th: degree bound c = 1} we have that ${}_1 \mathcal{B}_f (a, b) \leq d \cdot (d - 2)$.
        Since $\gcd \left( d, q \right) = 1$ we have that $d \cdot (d - 2)$ is odd.
        On the other hand, in characteristic $2$ the boomerang equations, see \Cref{Def: boomerang uniformity}, are symmetric.
        Thus, the number of solutions must be even, and the bound from \Cref{Th: degree bound c = 1} is not attainable.
    \end{proof}

    With \Cref{Prop: irreducible polynomial} we can also prove a special case for the $-1$-Boomerang uniformity.
    \begin{cor}\label[cor]{Cor: degree bound odd function c = -1}
        Let $\Fq$ be a finite field of odd characteristic, and let $f \in \Fq [X]$ with $d = \degree{f} > 0$ be a polynomial such that
        \begin{enumerate}[label=(\roman*)]
            \item $\gcd \left( d, q \right) = 1$, and

            \item $f (-x) = -f (x)$.
        \end{enumerate}
        Let $a, b \in \Fqx$.
        Then
        \begin{enumerate}
            \item $f (x + y) + f (x) + b$ is absolutely irreducible.

            \item ${}_{-1} \mathcal{B}_f (a, b) \leq d \cdot (d - 1)$.
        \end{enumerate}
    \end{cor}
    \begin{proof}
        We consider the $-1$-Boomerang equations with the substitution $z = x + y$
        \begin{align*}
            F_1 &= f (z) + f (x) - b, \\
            F_2 &= f (z + a) + f (x + a) - b.
        \end{align*}
        Next we perform the substitution $x = -\hat{x}$, then
        \begin{equation*}
            F_1 = f (z) - f (\hat{x}) - b,
        \end{equation*}
        and by \Cref{Lem: change of variables irreducible} and \Cref{Prop: irreducible polynomial} this polynomial is absolutely irreducible.

        Since $\degree{F_2 - F_1} \leq (d - 1) < d = \degree{F_1}$, these polynomials cannot share an irreducible component after homogenization.
        Now we can apply B\'ezout's theorem analog to \Cref{Th: degree bound c = 1} to yield the bound.
    \end{proof}

    \section{Arithmetic of \texorpdfstring{$c$}{c}-Boomerang Polynomial Systems}\label{Sec: Arithmetic of Boomerang Polynomial Systems}
    For $c = -1$ the $c$-Boomerang equations are symmetric, see \Cref{Def: boomerang uniformity}.
    In particular, their Newton polytopes are symmetric triangles.
    However, the GCD criterion for triangle indecomposability (\Cref{Cor: triangle indecomposable}) requires a non-symmetry in the triangle.\footnote{A bit more precise, the Newton polytope of $F = f (x) + f (y) + a$, where $d = \degree{f}$ and $f (0) = 0$, is given by $\big< \mathbf{0}, (d, 0)^\intercal, (0, d)^\intercal \big>$, and $\gcd \big( (d, 0)^\intercal, (0, d)^\intercal \big) = 1$.}
    Therefore, as an alternative approach to derive $c$-Boomerang uniformity bounds we study in this section the arithmetic of $c$-Boomerang polynomial systems.

    Recall that for $f \in k [x_1, \dots, x_n]$ we denote the highest homogeneous component of $f$ with $f^\topcomp$.
    Then for $\mathcal{F} \subset k [x_1, \dots, x_n]$ a finite system of polynomials, we denote $\mathcal{F}^\topcomp = \left\{ f^\topcomp \right\}_{f \in \mathcal{F}}$.
    Central to our arithmetic investigations are the following properties of the highest degree components ideal $\left( \mathcal{F}^\topcomp \right)$.
    \begin{lem}\label[lem]{Lem: vector space dimension}
        Let $k$ be a field, let $\mathcal{F} \subset P = k [x_1, \dots, x_n]$ be a finite system of polynomials, and let $d = \dim_k \big( P / (\mathcal{F}^\topcomp) \big)$.
        Then
        \begin{enumerate}
            \item $\big( \LM_{DRL} (f) \; \big\vert \; f \in (\mathcal{F}^\topcomp) \big) \subset \big( \LM_{DRL} (f) \; \big\vert \; f \in (\mathcal{F}) \big)$.

            \item $\dim_k \big( P / (\mathcal{F}) \big) \leq d$.
        \end{enumerate}
    \end{lem}
    \begin{proof}
        Let $h \in \left( \mathcal{F}^\topcomp \right)$ be non-zero, then we can express it as
        \begin{equation*}
            h = \sum_{f \in \mathcal{F}^\topcomp} g_f \cdot f^\topcomp,
        \end{equation*}
        where $g_f \in P$ is homogeneous and if it is non-zero then $\degree{g_f} = \degree{h} - \degree{f}$.
        We can lift $h$ to an element $\hat{h} \in (\mathcal{F})$ via
        \begin{equation*}
            \hat{h} = \sum_{f \in \mathcal{F}} g_f \cdot f,
        \end{equation*}
        and by homogeneity of the $g_f$'s we have that $\degree{\hat{h}} = \degree{h}$ and $\LM_{DRL} \left( \hat{h} \right) = \LM_{DRL} \left( h \right)$.
        This proves (1).

        For (2), if $d = \infty$ the inequality is trivial, otherwise we recall that one possible $k$-vector space basis of $P / (\mathcal{F})$ is given by all monomials not contained in $\big( \LM_{DRL} (f) \; \big\vert \; f \in (\mathcal{F}) \big)$ \cite[Chapter~5~\S 3~Proposition 4]{Cox-Ideals}.
        Via the inclusion from (1) we then have $\abs{P \setminus \big\{ \LM_{DRL} (f) \; \big\vert \; f \in (\mathcal{F}) \big\}} \leq \abs{P \setminus \big\{ \LM_{DRL} (f) \; \big\vert \; f \in (\mathcal{F}^\topcomp) \big\}} = d$.
    \end{proof}

    Thus, if we can construct $x_i^{d_i} \in \big( \LM_{DRL} (f) \; \big\vert \; f \in (\mathcal{F}^\topcomp) \big)$ for every variable $x_i$, then we have immediately established the following two properties:
    \begin{enumerate}
        \item Zero-dimensionality of $\left( \mathcal{F}^\topcomp \right)$ and $\left( \mathcal{F} \right)$, see \cite[Chapter~5~\S 3~Theorem~6]{Cox-Ideals} and \cite[Theorem 5.11]{Kemper-CommutativeAlgebra}.

        \item $\abs{\mathcal{V} (\mathcal{F}) \big( k \big)} = \dim_k \left( P / (\mathcal{F}) \right) \leq \dim_k \left( P / \left( \mathcal{F}^\topcomp \right) \right) \leq \prod_{i = 1}^{n} d_i$, see \cite[Chapter~5~\S 3~Proposition~7]{Cox-Ideals}.
    \end{enumerate}

    The highest degree components of the $c$-Boomerang polynomials are relatively simple bivariate polynomials, which allows construction of their DRL Gr\"obner bases by hand.

    \subsection{\texorpdfstring{$c^2 \neq 1$}{c\textasciicircum 2 != 1}}
    We start with the simplest case $c^2 \neq 1$, for which we can trivially compute a DRL Gr\"obner basis of the $c$-Boomerang polynomials.
    \begin{prop}\label[prop]{Prop: c-boomerang unfiformity}
        Let $\Fq$ be a finite field, let $f \in \Fq [X]$ with $d = \degree{f}$, let $c \in \Fqx$ be such that $c^2 \neq 1$, let $a \in \Fqx$ and $b \in \Fq$, and let $F_1, F_2 \in \Fq [x, z]$ be the $c$-Boomerang polynomials of $(a, b)$.
        Then
        \begin{enumerate}
            \item $\left\{ F_1, c^{-1} \cdot F_2 - F_1 \right\}$ is a DRL Gr\"obner basis of $(F_1, F_2)$ for $z >_{DRL} x$.

            \item ${}_c \mathcal{B}_f (a, b) \leq d^2$.
        \end{enumerate}
    \end{prop}
    \begin{proof}
        Starting from the $c$-Boomerang polynomials
        \begin{align*}
            F_1 &= f(z) - c \cdot f (x) - b, \\
            F_2 &= c \cdot f (z + a) - f (x + a) - b,
        \end{align*}
        we equip $\Fq [x, z]$ with the DRL term order $z >_{DRL} x$.
        Then $\LM_{DRL} \left( F_1 \right) = z^d$.
        Moreover,
        \begin{equation*}
            c^{-1} \cdot F_2 - F_1
            = \big( f (z - a) - f (z) \big) + c^{-1} \cdot \big( f (x + a) - c^2 \cdot f (x) \big) - b \cdot \left( 1 - c^{-1} \right).
        \end{equation*}
        Note that this polynomial does not have the monomial $z^d$, but since $c^2 \neq 1$ it has the monomial $x^d$.
        If a set of polynomials has pairwise coprime leading monomials, then it is already a Gr\"obner basis, see \cite[Chapter~2~\S 9~Theorem~3, Proposition~4]{Cox-Ideals}.
        Moreover, the number of monomials not contained in $(x^d, z^d) \in \Fq [x, z]$ is exactly $d^2$.
        By \cite[Chapter~5~\S 3~Proposition~7]{Cox-Ideals} this proves the claim.
    \end{proof}

    \subsection{\texorpdfstring{$c = - 1$}{c = -1}}
    Next we investigate the highest degree components for $c = -1$.
    \begin{prop}\label[prop]{Prop: highest degree components c = -1}
        Let $k$ be a field of characteristic $p \neq 2$, let $d \in \Z_{\geq 2}$, let $I = \big( x^d + y^d, x^{d - 1} + y^{d - 1} \big) \subset k [x, y]$, and let $x >_{DRL} y$.
        Then
        \begin{enumerate}
            \item $I$ has the DRL Gr\"obner basis
            \begin{equation*}
                \mathcal{G} = \{ g_1, g_2, g_3 \} = \left\{ x^{d - 1} + y^{d - 1} , x \cdot y^{d - 1} - y^d, y^{2 \cdot d - 2} \right\}.
            \end{equation*}

            \item $\dim_k \big( k [x, y] / I \big) = d \cdot (d - 1)$.
        \end{enumerate}
    \end{prop}
    \begin{proof}
        For (1), first we show that indeed $\mathcal{G} \subset I$.
        We have that
        \begin{equation*}
            g_2 = S_{DRL} \left( x^{d - 1} + y^{d - 1}, x^d + y^d \right) = x \cdot y^{d - 1} - y^d.
        \end{equation*}
        We claim that
        \begin{equation}\label{Equ: remainder claim c = -1}
            y^{d - 1} \cdot \left( x^{d - 1} + y^{d - 1} \right) \equiv 2 \cdot y^{2 \cdot d - 2} \mod \left( x \cdot y^{d - 1} - y^d \right).
        \end{equation}
        To prove this assertion we perform inductive division by remainder with respect to DRL on the terms on the left-hand side \Cref{Equ: remainder claim c = -1}.
        In the first step of the division algorithm we reduce
        \begin{equation*}
            x^{d - 1} \cdot y^{d - 1}
            = x^{d - 2} \cdot \left( x \cdot y^{d - 1} \right)
            \equiv x^{d - 2} \cdot y^d \mod \left( x \cdot y^{d - 1} - y^d \right).
        \end{equation*}
        Next we have to reduce
        \begin{equation*}
            x^{d - 2} \cdot y^d
            = x^{d - 3} \cdot y \cdot \left( x \cdot y^{d - 1} \right)
            \equiv x^{d - 3} \cdot y^{d + 1}
            \mod \left( x \cdot y^{d - 1} - y^d \right).
        \end{equation*}
        Per induction, in the $i$\textsuperscript{th} iteration we have to reduce
        \begin{equation*}
            x^{d - i} \cdot y^{d + i - 2}
            = x^{d - i - 1} \cdot y^{i - 1} \cdot \left( x \cdot y^{d - 1} \right)
            \equiv x^{d - i - 1} \cdot y^{d + i - 1} \mod \left( x \cdot y^{d - 1} - y^d \right).
        \end{equation*}
        This procedure will eventually terminate in $i = d - 1$ where we arrive at the term $y^{2 \cdot d - 2}$.
        Since $y^{2 \cdot d - 2}$ is also present in the left-hand side of \Cref{Equ: remainder claim c = -1} we obtain the final remainder $2 \cdot y^{2 \cdot d - 1} \neq 0$, since we work in characteristic $p \neq 2$.
        Note that this division process can also be expressed as
        \begin{equation*}
            y^{d - 1} \cdot \left( x^{d - 1} + y^{d - 1} \right) - \left( x \cdot y^{d - 1} - y^d \right) \cdot \left( \sum_{i = 0}^{d - 2} x^{d - 2 - i} \cdot y^i \right)
            = 2 \cdot y^{2 \cdot d - 2}.
        \end{equation*}
        It also immediately follows that $(\mathcal{G}) = I$.
        Next we verify Buchberger's criterion (\Cref{Th: Buchberger criterion}):
        \begin{itemize}
            \item $S_{DRL} \left( g_1, g_3 \right)$ can be ignored due to coprimality of the leading monomials, see \cite[Chapter~2~\S 9~Theorem~3, Proposition~4]{Cox-Ideals}.

            \item The remainder of $S_{DRL} \left( g_1, g_2 \right)$ with respect to $\mathcal{G}$ corresponds exactly to the polynomial division of \Cref{Equ: remainder claim c = -1}, i.e.\ $S_{DRL} (g_1, g_2) \equiv 2 \cdot g_3 \equiv 0 \mod (\mathcal{G})$.

            \item For the last pair
            \begin{equation*}
                S_{DRL} \left( g_2, g_3 \right)
                = y^{d - 1} \cdot \left( x \cdot y^{d - 1} - y^d \right) - x \cdot y^{2 \cdot d - 2}
                = -y^{2 \cdot d - 1}
                \equiv 0 \mod \left( y^{2 \cdot d - 2} \right).
            \end{equation*}
        \end{itemize}
        So, $\mathcal{G}$ is indeed the DRL Gr\"obner basis of $I$.

        For (2), by \cite[Chapter~5~\S 3~Proposition~7]{Cox-Ideals} we have to count the monomials not contained in $\left( x^{d - 1}, x \cdot y^{d - 1}, y^{2 \cdot d - 2} \right) \subset k [x, y]$.
        There are
        \begin{itemize}
            \item $d - 2$ monomials $x, \dots, x^{d - 2}$,

            \item $(d - 2)^2$ monomials $x^k \cdot y^l$, where $1 \leq k, l \leq d - 2$, and

            \item $2 \cdot (d - 1) - 1$ monomials $y, \dots, y^{2 \cdot d - 3}$.
        \end{itemize}
        So, incorporating the constant monomial $1$ we have the $k$-vector space dimension $1 + d - 2 + (d - 2)^2 + 2 \cdot (d - 1) - 1 = d \cdot (d - 1)$.
    \end{proof}

    The homogeneous ideal from \Cref{Prop: highest degree components c = -1} is exactly the highest degree components ideal for $c = -1$ after an obvious preprocessing step.
    \begin{thm}\label[thm]{Th: degree bound c = -1}
        Let $\Fq$ be a finite field of odd characteristic, let $f \in \Fq [x]$ with $d = \degree{f}$ be such that $\gcd \left( d, q \right) = 1$, and let $a \in \Fqx$ and $b \in \Fq$.
        Then ${}_{-1} \mathcal{B}_f (a, b) \leq d \cdot (d - 1)$.
    \end{thm}
    \begin{proof}
        Let $f = \sum_{i = 0}^{d} a_i \cdot x^i$.
        We consider the $-1$-Boomerang equations with the substitution $z = x + y$
        \begin{align*}
            F_1 &= f (z) + f (x) - b, \\
            F_2 &= f (z + a) + f (x + a) - b,
        \end{align*}
        as polynomials in $\Fq [x, z]$.
        Observe that
        \begin{equation*}
            \Delta_a f (x)
            = f (x + a) - f (x)
            = \sum_{i = 0}^{d} a_i \cdot \left( (x + a)^i - x^i \right)
            = a_d \cdot d \cdot a \cdot x^{d - 1} + \ldots,
        \end{equation*}
        and since by assumption $\gcd \left( d, q \right) = 1$ this is a polynomial of degree $d - 1$.
        Now let
        \begin{equation*}
            G_2 = F_2 - F_1 = \Delta_a f (z) + \Delta_a f (x).
        \end{equation*}
        Then $\left( F_1^\topcomp, G_2^\topcomp \right) = \left( z^d + x^d, z^{d - 1} + x^{d - 1} \right)$, and the claim follows from \Cref{Prop: highest degree components c = -1}.
    \end{proof}

    \subsection{\texorpdfstring{$c = 1$}{c = 1}}
    For completeness, we also reinvestigate the case $c = 1$.
%
    \begin{prop}\label[prop]{Prop: highest degree components c = 1}
        Let $k$ be a field of characteristic $p$, let $d \in \Z_{\geq 3}$, let $I = \Big( x^d - y^d, \allowbreak \sum_{i + j = d - 2} x^i \cdot y^j \Big) \subset k [x, y]$, and let $x >_{DRL} y$.
        If $p > 0$, then assume in addition that $\gcd \left( d - 1, p \right) = 1$.
        Then
        \begin{enumerate}
            \item $I$ has the DRL Gr\"obner basis
            \begin{equation*}
                \mathcal{G} = \{ g_1, g_2, g_3 \} = \left\{ \sum_{i + j = d - 2} x^i \cdot y^j, x \cdot y^{d - 1} - y^d, y^{2 \cdot d - 3} \right\}.
            \end{equation*}

            \item $\dim_k \big( k [x, y] / I \big) = d \cdot (d - 2)$.
        \end{enumerate}
    \end{prop}
    \begin{proof}
        For (1), note that
        \begin{equation*}
            (x - y) \cdot \left( \sum_{i + j = d - 2} x^i \cdot y^j \right) = x^{d - 1} - y^{d - 1},
        \end{equation*}
        see for example the proof of \cite[Lemma~3.5]{Steiner-Boomerang}, and
        \begin{equation*}
            g_2 = S_{DRL} \left( x^d - y^d, x^{d - 1} - y^{d - 1} \right) = x \cdot y^{d - 1} - y^d.
        \end{equation*}
        We claim that
        \begin{equation}\label{Equ: remainder claim c = 1}
            y^{d - 1} \cdot \left( \sum_{i + j = d - 2} x^i \cdot y^j \right) \equiv (d - 1) \cdot y^{2 \cdot d - 2} \mod \left( x \cdot y^{d - 1} - y^d \right).
        \end{equation}
        To prove this assertion we perform inductive division by remainder with respect to DRL on the terms on the left-hand side \Cref{Equ: remainder claim c = 1}.
        In the first step of the division algorithm we reduce
        \begin{equation*}
            x^{d - 2} \cdot y^{d - 1}
            = x^{d - 3} \cdot \left( x \cdot y^{d - 1} \right)
            \equiv x^{d - 3} \cdot y^d \mod \left( x \cdot y^{d - 1} - y^d \right).
        \end{equation*}
        Since $y^{d - 1} \cdot \left( x^{d - 3} \cdot y \right) = x^{d - 3} \cdot y^d$ is also present in the left-hand side of \Cref{Equ: remainder claim c = 1}, in the second step we have to reduce
        \begin{equation*}
            2 \cdot x^{d - 3} \cdot y^d
            = 2 \cdot x^{d - 4} \cdot y \cdot \left( x \cdot y^{d - 1} \right)
            \equiv 2 \cdot x^{d - 4} \cdot y^{d + 1}
            \mod \left( x \cdot y^{d - 1} - y^d \right).
        \end{equation*}
        Per induction, in the $i$\textsuperscript{th} iteration we have to reduce
        \begin{equation*}
            i \cdot x^{d - i - 1} \cdot y^{d + i - 2}
            = i \cdot x^{d - i - 2} \cdot y^{i - 1} \cdot \left( x \cdot y^{d - 1} \right)
            \equiv i \cdot x^{d - i - 2} \cdot y^{d + i - 1} \mod \left( x \cdot y^{d - 1} - y^d \right).
        \end{equation*}
        This procedure will eventually terminate in $i = d - 2$ where we arrive at the term $(d - 2) \cdot y^{2 \cdot d - 3}$.
        Since $y^{2 \cdot d - 3}$ is also present in the left-hand side of \Cref{Equ: remainder claim c = 1} we obtain the final remainder $(d - 1) \cdot y^{2 \cdot d - 3}$.
        If $p > 0$, then by assumption $d - 1 \not\equiv 0 \mod p$, so the final remainder is indeed non-zero which implies $\mathcal{G} \subset I$.
        Note that this division process can also be expressed as
        \begin{equation*}
            y^{d - 1} \cdot \left( \sum_{i + j = d - 2} x^i \cdot y^j \right) - \left( x \cdot y^{d - 1} - y^d \right) \cdot \left( \sum_{i = 0}^{d - 3} (i + 1) \cdot x^{d - 3 - i} \cdot y^i \right)
            = (d - 1) \cdot y^{2 \cdot d - 3}.
        \end{equation*}
        It also immediately follows that $(\mathcal{G}) = I$.
        Next we verify Buchberger's criterion (\Cref{Th: Buchberger criterion}):
        \begin{itemize}
            \item $S_{DRL} \left( g_1, g_3 \right)$ can be ignored due to coprimality of the leading monomials, see \cite[Chapter~2~\S 9~Theorem~3, Proposition~4]{Cox-Ideals}.

            \item $S_{DRL} (g_1, g_2)$ corresponds exactly to the first step in the polynomial division of \Cref{Equ: remainder claim c = 1}, therefore $S_{DRL} (g_1, g_2) \allowbreak \equiv (d - 1) \cdot g_3 \equiv 0 \mod \left( \mathcal{G} \right)$.

            \item  For the last pair
            \begin{equation*}
                S_{DRL} \left( g_2, g_3 \right)
                = y^{d - 2} \cdot \left( x \cdot y^{d - 1} - y^d \right) - x \cdot y^{2 \cdot d - 3}
                = -y^{2 \cdot d - 2}
                \equiv 0 \mod \left( y^{2 \cdot d - 3} \right).
            \end{equation*}
        \end{itemize}
        So, $\mathcal{G}$ is indeed the DRL Gr\"obner basis of $I$.

        For (2), by \cite[Chapter~5~\S 3~Proposition~7]{Cox-Ideals} we have to count the monomials not contained in $\left( x^{d - 2}, x \cdot y^{d - 1}, y^{2 \cdot d - 3} \right) \subset k [x, y]$.
        There are
        \begin{itemize}
            \item $d - 3$ monomials $x, \dots, x^{d - 3}$,

            \item $(d - 3) \cdot (d - 2)$ monomials $x^k \cdot y^l$, where $1 \leq k \leq d - 3$ and $1 \leq l \leq d - 2$, and

            \item $2 \cdot d - 4$ monomials $y, \dots, y^{2 \cdot d - 4}$.
        \end{itemize}
        So, incorporating the constant monomial $1$ we have the $k$-vector space dimension $1 + d - 3 + (d - 3) \cdot (d - 2) + 2 \cdot d - 4 = d \cdot (d - 2)$.
    \end{proof}

    If we apply \Cref{Prop: highest degree components c = 1} to the $1$-Boomerang polynomials we obtain a slightly more restricted bound than \Cref{Th: degree bound c = 1}.
    \begin{cor}\label[cor]{Cor: degree bound c = 1}
        Let $\Fq$ be a finite field, and let $f \in \Fq [x]$ with $d = \degree{f}$ be such that $\gcd \left( d, q \right) = \gcd \left( d - 1, q \right) = 1$, and let $a \in \Fqx$ and $b \in \Fq$.
        Then ${}_1 \mathcal{B}_f (a, b) \leq d \cdot (d - 2)$.
    \end{cor}
    \begin{proof}
        Let $f = \sum_{i = 0}^{d} a_i \cdot x^i$.
        We consider the $-1$-Boomerang equations with the substitution $z = x + y$
        \begin{align*}
            F_1 &= f (z) - f (x) - b, \\
            F_2 &= f (z + a) - f (x + a) - b,
        \end{align*}
        as polynomials in $\Fq [x, z]$.
        Observe that
        \begin{equation*}
            \Delta_a f (x)
            = f (x + a) - f (x)
            = \sum_{i = 0}^{d} a_i \cdot \left( (x + a)^i - x^i \right)
            = a_d \cdot d \cdot a \cdot x^{d - 1} + \ldots,
        \end{equation*}
        and since by assumption $\gcd \left( d, q \right) = 1$ this is a polynomial of degree $d - 1$.
        Now let
        \begin{equation*}
            \tilde{F}_2 = F_2 - F_1 = \Delta_a f (z) - \Delta_a f (x).
        \end{equation*}
        As discussed in the proof of \Cref{Th: degree bound c = 1}, $(z - x) \mid \tilde{F}_2$ and we can replace it by $G_2 = \frac{\tilde{F}_2}{z - x}$, which has degree $d - 2$.
        In particular, $(F_1^\topcomp, G_2^\topcomp) = \left( z^d - x^d, \sum_{i + j = d - 2} z^{i} \cdot x^{j} \right)$, this again follows from the proof of \cite[Lemma~3.5]{Steiner-Boomerang}, and we obtain the inequality via \Cref{Prop: highest degree components c = 1}.
    \end{proof}

    On the other hand, the improved bound in characteristic $2$ from \Cref{Cor: binary degree bound c = 1} is not covered by \Cref{Prop: highest degree components c = 1} respectively \Cref{Cor: degree bound c = 1}.

    \section{Constructing Tight Examples}\label{Sec: Constructing Tight Examples}
    With the aid of \SageMath \cite{SageMath} we now construct tight examples for \Cref{Th: degree bound c = 1}, \Cref{Cor: binary degree bound c = 1}, \Cref{Prop: c-boomerang unfiformity} and \Cref{Th: degree bound c = -1}.
    For $f \in \Fq [X]$ and $c \in \Fqx$, our tight examples are found via the following procedure:
    \begin{enumerate}
        \item Pick $a, b \in \Fqx$ and set up the $c$-Boomerang polynomials $F_1, F_2 \in \Fq [x, z]$.

        \item Compute the DRL Gr\"obner basis of $(F_1, F_2)$ for $z >_{DRL} x$.

        Compute $D = \dim_{\Fq} \big( \Fq [x, z] / (F_1, F_2) \big)$.
        If $D$ is less than the expected bound, pick different $(a, b)$.

        \item Compute the LEX Gr\"obner basis for $z >_{LEX} x$ via term order conversion \cite{Faugere-FGLM}.

        If the LEX basis is not of the form $\big\{ z - g_1 (x), g_2 (x) \big\}$, pick different $(a, b)$.

        \item Factor $g_2 (x)$.

        If every factor has multiplicity we pass to an appropriate field extension $\Fq \subset \F_{q^n}$ such that $g_2$ decomposes into distinct linear factors.
        Otherwise, pick different $(a, b)$.
    \end{enumerate}

    In \Cref{Alg: polynomial generation} we provide the \SageMath function for generation of the $c$-Boomerang polynomials.
    \begin{algorithm}[H]
        \centering
        \caption{Construction of $c$-Boomerang Polynomials.}
        \label{Alg: polynomial generation}
\begin{lstlisting}[language=MyPython]
def c_boomerang_polynomials(f,    # A polynomial in Fq [X].
                            a,    # An element in Fq.
                            b,    # An element in Fq.
                            c=1,  # An element in Fq.
                           ):
    P = PolynomialRing(f.parent().base_ring(), ["z", "x"], order="degrevlex")
    z = P.gens()[0]
    x = P.gens()[1]
    F_1 = f(z) - c * f(x) - b
    F_2 = c * f(z + a) - f(x + a) - c * b
    return [F_1, F_2]
\end{lstlisting}
    \end{algorithm}

    \subsection{A Short Excursion on Dickson Polynomials}
    Dickson polynomials have numerous applications in the study of finite fields.
    In particular, they provide a generic way to instantiate permutation polynomials.
    All our tight examples in this section are instantiations of Dickson polynomials.
    \begin{defn}[{\cite{Lidl-Dickson}}]
        Let $\Fq$ be a finite field, let $n \in \Z_{\geq 0}$, and let $a \in \Fq$.
        The $n$\textsuperscript{th} Dickson polynomial (of the first kind) is defined as
        \begin{equation*}
            D_n (X, a) = \sum_{i = 0}^{\floor{\frac{n}{2}}} \frac{n}{n - i} \cdot \binom{n - i}{i} \cdot (-a)^i \cdot X^{n - 2 \cdot i} \in \Fq [X].
        \end{equation*}
    \end{defn}
    The following properties of Dickson polynomials are well-known:
    \begin{itemize}
        \item $D_n (X, 0) = X^n$.

        \item For $n \geq 2$ they satisfy the recurrence relation \cite[Lemma~2.3]{Lidl-Dickson}
        \begin{equation}
            D_n (X, a) = X \cdot D_{n - 1} (X, a) - a \cdot D_{n - 2} (X, a).
        \end{equation}

        \item $D_n (X, a)$ is a permutation polynomial over $\Fq$ if and only if $\gcd \left( n, q^2 - 1 \right) = 1$ \cite[7.16.~Theorem]{Niederreiter-FiniteFields}.

        \item They satisfy a commutation relation under composition \cite[Lemma~2.6]{Lidl-Dickson}
        \begin{equation}\label{Equ: Dickson commutation}
            D_{m \cdot n} (X, a) = D_m \big( D_n (X, a), a^n \big) = D_n \big( D_m (X, a), a^m \big).
        \end{equation}

        \item They satisfy \cite[Lemma~2.3]{Lidl-Dickson}
        \begin{equation}\label{Equ: Dickson polynomial multiplication with constant}
            b^n \cdot D_n (X, a) = D_n (b \cdot X, b^2 \cdot a).
        \end{equation}

        \item For $p$ the characteristic of $\Fq$ they satisfy \cite[Lemma~2.3]{Lidl-Dickson}
        \begin{equation}
            D_{p \cdot n} (X, a) = \big( D_n (X, a) \big)^p.
        \end{equation}

        \item If $2 \mid d$, then all terms on $D_d (X, a)$ have even degree.

        \item If $2 \nmid d$, then all terms on $D_d (X, a)$ have odd degree.
    \end{itemize}

    Thus, given a Dickson polynomial of degree $p^k \cdot d$ where $p \nmid d$ and $d > 1$, then we can reduce the $c$-Boomerang equations of $D_{p^k \cdot d} (X, a)$ to $c$-Boomerang equations of $D_d (X, a)$ via an application of a linearized permutation monomial.
    Thus, we can always apply \Cref{Th: degree bound c = 1,Th: degree bound c = -1} to estimate the $\pm 1$-Boomerang uniformity of Dickson polynomials.

    In \Cref{Alg: Dickson polynomial generation} we provide the \SageMath function for generation of Dickson polynomials.
    \begin{algorithm}[H]
        \centering
        \caption{Construction of $n$\textsuperscript{th} Dickson polynomial.}
        \label{Alg: Dickson polynomial generation}
\begin{lstlisting}[language=MyPython]
def dickson_polynomial(n, # Degree of the Dickson polynomial.
                       a, # Coefficient of the Dickson polynomial,
                          # must be an element of a SageMath ring.
                      ):
    ring = a.parent()
    P = PolynomialRing(ring, 'X')
    X = P.gen()
    p = 0
    for i in range(0, floor(n / 2) + 1):
        p += Integer((n / (n - i)) * binomial(n - i, i)) * ring(-a)**i * X**(n - 2 * i)
    return p
\end{lstlisting}
    \end{algorithm}

    The $c$-BCT of Dickson polynomials satisfies the following relation.
    \begin{lem}\label[lem]{Lem: Dickson polynomials square relationship}
        Let $\Fq$ be a finite field, let $n \in \Z_{\geq 1}$, let $\alpha, \beta \in \Fqx$, and let $\gamma = \frac{\beta}{\alpha}$.
        If either
        \begin{enumerate}[label=(\roman*)]
            \item $\alpha$ and $\beta$ are squares, or

            \item $\alpha$ and $\beta$ are non-squares,
        \end{enumerate}
        then ${}_c \mathcal{B}_{D_n (x, \alpha)} (a, b) = {}_c \mathcal{B}_{D_n (x, \beta)} \big( \gamma^\frac{1}{2} \cdot a, \gamma^\frac{n}{2} \cdot b \big)$.
        In particular, $\beta_{D_n (x, \alpha), c} = \beta_{D_n (x, \beta), c}$.
    \end{lem}
    \begin{proof}
        \begin{itemize}
            \item []

            \item If $\alpha$ and $\beta$ are squares, then $\gamma$ is one too.

            \item If $\alpha$ and $\beta$ are non-squares, then $\gamma$ is a square.\footnote{A very elegant argument for the product of non-squares in $\Fq$ being square which does not require cyclicity of $\Fqx$ was given in \url{https://math.stackexchange.com/a/1938451/696633}.}
        \end{itemize}
        Therefore, in either case we can utilize \Cref{Equ: Dickson polynomial multiplication with constant} as follows
        \begin{equation*}
            \gamma^\frac{n}{2} \cdot D_n (x, \alpha) = D_n \Big( \gamma^\frac{1}{2} \cdot x, \gamma \cdot \alpha \Big) = D_n \Big( \gamma^\frac{1}{2} \cdot x, \beta \Big).
        \end{equation*}
        From \Cref{Def: boomerang uniformity} it then follows that ${}_c \mathcal{B}_{D_n (x, \alpha)} (a, b) = {}_c \mathcal{B}_{D_n (x, \beta)} \big( \gamma^\frac{1}{2} \cdot a, \gamma^\frac{n}{2} \cdot b \big)$.
    \end{proof}

    This lemma has an interesting consequence.
    If we consider $\left\{ \alpha, \beta_{D_n (x, \alpha), c} \right\}_{\alpha \in \Fq}$ as table, then we can have at most three distinct entries for $\beta_{D_n (x, \alpha), c}$.
    Namely, one entry for $\alpha = 0$, one for $\alpha$ square and one for $\alpha$ non-square.
    \begin{ex}
        Let $q = 11$, then we have the following $c$-Boomerang uniformities for $D_7 (X, \alpha)$.
        \begin{table}[H]
            \centering
            \caption{$c$-Boomerang uniformities of $D_7 (X, \alpha)$ over $\F_{11}$.}
            \begin{tabular}{ c | c | c | c | c | c | c | c | c | c | c }
                \toprule
                & \multicolumn{10}{ c }{$\beta_{D_7 (x, \alpha), c}$} \\
                \midrule
                $c$ & $1$ & $2$ & $3$ & $4$ & $5$ & $6$ & $7$ & $8$ & $9$ & $10$  \\
                \midrule
                $\alpha = 0$        & $3$ & $5$ & $2$ & $2$ & $2$ & $5$ & $2$ & $2$ & $2$ & $4$ \\
                $\alpha$ square     & $6$ & $3$ & $4$ & $4$ & $3$ & $3$ & $5$ & $5$ & $3$ & $6$ \\
                $\alpha$ non-square & $5$ & $5$ & $3$ & $3$ & $3$ & $5$ & $3$ & $3$ & $3$ & $4$ \\
                \bottomrule
            \end{tabular}
        \end{table}
        The \SageMath code to compute the $c$-Boomerang uniformities of Dickson polynomials over $\F_{11}$ can be found in \Cref{Alg: Dickson c-Boomerang uniformities}.
    \end{ex}

    \subsection{Examples for \texorpdfstring{$c = 1$}{c = 1}}
    First, we present a tight example for \Cref{Th: degree bound c = 1} in odd characteristic.
    \begin{ex}\label[ex]{Ex: tight example q = 11}
        Let $q = 11$, let $f = X^7$, and let $a = 1$ and $b = 3$.
        Then, the DRL Gr\"obner basis $z >_{DRL} x$ of the boomerang equations is
        \begin{align*}
            g_1 &=
            &&x^{11} + 4 z^{4} x^{5} + 8 x^{9} - z^{4} x^{4} + 8 z^{3} x^{5} + 3 x^{8} - z^{4} x^{3} + 4 z^{3} x^{4} + 8 z^{2} x^{5} + 7 x^{7} + 5 z^{4} x^{2} - z^{3} x^{3} + 7 z^{2} x^{4} \\
            &\phantom{=}
            &&+ 4 z x^{5} + 5 x^{6} + 5 z^{4} x + z^{3} x^{2} + 6 z^{2} x^{3} + z x^{4} - x^{5} + 8 z^{4} + 7 z^{3} x + 5 z^{2} x^{2} + 6 z x^{3} + 5 x^{4} + 8 z^{3} - z^{2} x \\
            &\phantom{=}
            &&+ 7 z x^{2} + 3 x^{3} + 7 z^{2} + 9 z x + 8 x^{2} + 6 z + 3 x + 1,
            \\
            g_2 &=
            &&z x^{6} - x^{7} + 3 z x^{5} + 8 x^{6} + 7 z^{4} x + 7 z^{3} x^{2} + 7 z^{2} x^{3} + z x^{4} + 9 x^{5} + 9 z^{4} - z^{3} x - z^{2} x^{2} + 4 z x^{3} + 6 x^{4} + 3 z^{3} \\
            &\phantom{=}
            &&+ 2 z^{2} x + 5 z x^{2} + 9 x^{3} - z^{2} + 3 z x + 4 x^{2} + 2 z + 7 x + 4,
            \\
            g_3 &=
            &&z^{5} + z^{4} x + z^{3} x^{2} + z^{2} x^{3} + z x^{4} + x^{5} + 3 z^{4} + 3 z^{3} x + 3 z^{2} x^{2} + 3 z x^{3} + 3 x^{4} + 5 z^{3} + 5 z^{2} x \\
            &\phantom{=}
            &&+ 5 z x^{2} + 5 x^{3} + 5 z^{2} + 5 z x + 5 x^{2} + 3 z + 3 x + 1,
        \end{align*}
        and the LEX Gr\"obner basis with $z >_{LEX} x$ is
        \begin{align*}
            h_1 &=
            &&z + 5 x^{34} + 4 x^{33} + 5 x^{32} + 4 x^{31} + 6 x^{30} + x^{29} + 8 x^{28} + x^{26} + 9 x^{25} + 4 x^{23} - x^{22} + 7 x^{20} - x^{19} + 4 x^{18} \\
            &\phantom{=}
            &&+ 6 x^{17} - x^{16} + x^{15} + 9 x^{14} + 4 x^{13} + 9 x^{12} + 2 x^{11} + 3 x^{10} + 2 x^{9} + 8 x^{8} + 6 x^{7} + 6 x^{6} + 2 x^{5} + 5 x^{3} + 4 x^{2} \\
            &\phantom{=}
            &&+ 6 x + 6,
            \\
            h_2 &=
            &&(x + 7) \cdot (x + 10) \cdot (x + 4) \cdot (x^{2} + 6 x + 1) \cdot (x^{2} + 8 x + 3) \cdot (x^{4} - x^{3} + 8 x^{2} + 4 x + 3) \\
            &\phantom{=}
            &&\cdot (x^{4} - x^{3} + 9 x^{2} + 7) \cdot (x^{20} + x^{19} + x^{18} - x^{17} + 4 x^{16} + 3 x^{15} + 6 x^{14} - x^{13} + x^{12} + 5 x^{11} + 9 x^{10} + x^{9}
            \\
            &\phantom{=}
            &&\phantom{\cdot (}
            + 7 x^{8} + 3 x^{6} + 7 x^{5} + 8 x^{4} + 6 x^{3} + 7 x^{2} + 9 x + 3).
        \end{align*}
        Therefore, over the finite field of size $q^{20}$ the univariate LEX polynomial will decompose into linear factors and \Cref{Th: degree bound c = 1} becomes tight.
        Also, note that $\gcd \left( 7, q^{20} - 1 \right) = 1$.
        The code for this example is provided in \Cref{Alg: tight example q = 11}.
    \end{ex}

    Next we present a tight example for \Cref{Cor: binary degree bound c = 1} in characteristic $2$.
    \begin{ex}\label[ex]{Ex: tight example q = 16}
        Let $\F_{2^4} (g) \cong \F_2 [Y] / \left( Y^4 + Y + 1 \right)$, let $f = D_{11} (X, 1)$, and let $a = g^2 + 1$ and $b = g^3 + g^2 + g$.
        Then, the LEX Gr\"obner basis $z >_{LEX} x$ of the boomerang equations is of the form
        \begin{align*}
            h_1 &= z + \tilde{h}_{1, 1} (x), \\
            h_2 &= \prod_{i = 1}^{9} \tilde{h}_{2, i} (x),
        \end{align*}
        where $\degree{\tilde{h}_{1, 1}} = 96$, $\degree{\tilde{h}_{2, 1}} = \ldots = \degree{\tilde{h}_{2, 3}} = 2$, $\degree{\tilde{h}_{2, 4}} = \degree{\tilde{h}_{2, 5}} = 4$, $\degree{\tilde{h}_{2, 6}} = \degree{\tilde{h}_{2, 7}} = 14$ and $\degree{\tilde{h}_{2, 8}} = \degree{\tilde{h}_{2, 9}} = 28$.
        So, $\dim_{\Fq} \left( h_1, h_2 \right) = 98 = 11 \cdot (11 - 2) - 1$, and over the finite field of size $q^{14} = 2^{56}$ the univariate LEX polynomial decomposes into linear factors and \Cref{Cor: binary degree bound c = 1} becomes tight.
        Also, note that $\gcd \left( 11, 2^{2 \cdot 56} - 1 \right) = 1$.
        The code for this example is provided in \Cref{Alg: tight example q = 16}.
    \end{ex}

    \subsection{Examples for \texorpdfstring{$c = -1$}{c = -1}}
    Now we present a tight example for \Cref{Th: degree bound c = -1} in odd characteristic.
    \begin{ex}\label[ex]{Ex: tight example q = 257}
        Let $q = 257$, let $f = X^5$, and let $a = 1$ and $b = 48$.
        Then, the DRL Gr\"obner basis $z >_{DRL} x$ of the $-1$-Boomerang equations is
        \begin{align*}
            g_1 &=
            x^{8} + 4 x^{7} + 2 z^{3} x^{3} + 8 x^{6} + 3 z^{3} x^{2} + 3 z^{2} x^{3} + 10 x^{5} + 156 z^{3} x + 158 z^{2} x^{2} + 156 z x^{3} + 214 x^{4} + 182 z^{3}
            \\
            &\phantom{=}
            + 235 z^{2} x + 235 z x^{2} + 84 x^{3} + 158 z^{2} + 210 z x + 57 x^{2} + 55 z + 4 x + 120, \\
            g_2 &=
            z x^{4} - x^{5} + 2 z x^{3} + 255 x^{4} + 255 z^{3} + 2 z x^{2} + 253 x^{3} + 254 z^{2} + z x + 253 x^{2} + 204 z + 255 x + 150, \\
            g_3 &=
            z^{4} + x^{4} + 2 z^{3} + 2 x^{3} + 2 z^{2} + 2 x^{2} + z + x + 206,
        \end{align*}
        and the LEX Gr\"obner basis with $z >_{LEX} x$ is
        \begin{align*}
            h_1 &=
            z + 198 x^{19} + 33 x^{18} + 165 x^{17} + 77 x^{16} + 45 x^{15} + 84 x^{14} + 122 x^{13} + 190 x^{12} + 166 x^{11} + 12 x^{10} + 79 x^{9}
            \\
            &\phantom{=}
            + 2 x^{8} + 120 x^{7} + 100 x^{6} + 239 x^{5} + 61 x^{4} + 96 x^{3} + 59 x^{2} + 141 x + 161,
            \\
            h_2 &=
            (x^{2} + 9 x + 3) \cdot (x^{2} + 25 x + 255) \cdot (x^{2} + 33 x + 132) \cdot (x^{2} + 50 x + 24) \cdot (x^{2} + 55 x + 197)
            \\
            &\phantom{=(}
            \cdot (x^{2} + 95 x + 164) \cdot (x^{2} + 134 x + 125) \cdot (x^{2} + 192 x + 35) \cdot (x^{2} + 219 x + 41) \cdot (x^{2} + 226 x + 86).
        \end{align*}
        Therefore, over the finite field of size $q^2$ the univariate LEX polynomial will decompose into linear factors and \Cref{Th: degree bound c = -1} becomes tight.
        Also, note that $\gcd \left( 5, q - 1 \right) = \gcd \left( 5, q^2 - 1 \right) = 1$.
        The code for this example is provided in \Cref{Alg: tight example q = 257}.
    \end{ex}

    We also provide a Dickson polynomial example.
    \begin{ex}\label[ex]{Ex: tight example q = 89}
        Let $p = 89$, let $f = D_7 (X, 1)$, and let $a = 1$ and $b = 17$.
        Then, the DRL Gr\"obner basis $z >_{DRL} x$ of the $-1$-Boomerang equations is
        \begin{align*}
            g_1 &=
            x^{12} + 6 x^{11} + 2 z^{5} x^{5} + 9 x^{10} + 5 z^{5} x^{4} + 5 z^{4} x^{5} + 79 x^{9} + 15 z^{4} x^{4} + 57 x^{8} + 84 z^{5} x^{2} + 5 z^{4} x^{3} + 5 z^{3} x^{4}
            \\
            &\phantom{=(}
            + 84 z^{2} x^{5} + 87 x^{7} + 37 z^{5} x + 24 z^{4} x^{2} + 49 z^{3} x^{3} + 24 z^{2} x^{4} + 37 z x^{5} + 50 x^{6} + 11 z^{5} + 29 z^{4} x + 34 z^{3} x^{2}
            \\
            &\phantom{=(}
            + 34 z^{2} x^{3} + 29 z x^{4} + 60 x^{5} + 20 z^{4} + 10 z^{3} x + 35 z^{2} x^{2} + 10 z x^{3} + x^{4} + 12 z^{3} + 5 z^{2} x + 5 z x^{2} + 32 x^{3}
            \\
            &\phantom{=(}
            + 34 z^{2} + 7 z x + 25 x^{2} + 66 z + 3 x + 24,
            \\
            g_2 &=
            z x^{6} - x^{7} + 3 z x^{5} + 86 x^{6} + 87 z^{5} + 87 x^{5} + 84 z^{4} + 84 z x^{3} - z x^{2} + x^{3} + 5 z^{2} + 2 z x + 3 x^{2} + 14 z
            \\
            &\phantom{=(}
            + x + 67,
            \\
            g_3 &=
            z^{6} + x^{6} + 3 z^{5} + 3 x^{5} + 84 z^{3} + 84 x^{3} - z^{2} - x^{2} + 2 z + 2 x + 13,
        \end{align*}
        and the LEX Gr\"obner basis with $z >_{LEX} x$ is
        \begin{align*}
            h_1 &=
            z + 71 x^{41} + 86 x^{40} + 41 x^{39} + 73 x^{38} + 82 x^{37} + 20 x^{36} + 37 x^{35} + 30 x^{34} + 65 x^{33} + 56 x^{32} + 51 x^{31} + 63 x^{30}
            \\
            &\phantom{=(}
            + 11 x^{29} - x^{28} + 32 x^{27} + 25 x^{26} + 56 x^{25} + 27 x^{24} + 24 x^{23} + 22 x^{22} + 58 x^{21} + 25 x^{20} + 42 x^{19} + 85 x^{18}
            \\
            &\phantom{=(}
            + 45 x^{17} + 59 x^{16} + 55 x^{15} + 84 x^{14} + 37 x^{12} + 79 x^{11} + 85 x^{10} + 71 x^{9} + 47 x^{8} + 16 x^{7} + 49 x^{6} + 75 x^{5} \\
            &\phantom{=(}
            + 66 x^{4} + 86 x^{3} + 55 x^{2} + 77 x + 28,
            \\
            h_2 &=
            (x + 28) \cdot (x + 34) \cdot (x + 73) \cdot (x + 80) \cdot (x + 85) \cdot x \cdot (x + 9) \cdot (x + 60) \cdot (x^{2} + 32 x + 55)
            \\
            &\phantom{=(}
            \cdot (x^{2} + 72 x + 13) \cdot (x^{10} + 26 x^{9} + 79 x^{8} + 63 x^{7} + 45 x^{6} + 64 x^{5} + 15 x^{4} + 16 x^{3} + 44 x^{2} + 11 x + 20)
            \\
            &\phantom{=(}
            \cdot (x^{10} + 62 x^{9} + 31 x^{8} + 3 x^{7} + 16 x^{6} + 69 x^{5} + 84 x^{4} + 73 x^{3} + 16 x^{2} + 50 x + 36)
            \\
            &\phantom{=(}
            \cdot (x^{10} + 83 x^{9} + 16 x^{8} + 30 x^{7} + 4 x^{6} + 27 x^{5} + 6 x^{4} + 67 x^{3} + 69 x^{2} + 65 x + 88).
        \end{align*}
        Therefore, over the finite field of size $q^10$ the univariate LEX polynomial will decompose into linear factors and \Cref{Cor: degree bound odd function c = -1} becomes tight.
        Also, note that $\gcd \left( 7, q^2 - 1 \right) = \gcd \left( 7, q^{20} - 1 \right) = 1$.
        The code for this example is provided in \Cref{Alg: tight example q = 89}.
    \end{ex}

    \subsection{Example for \texorpdfstring{$c^2 \neq 1$}{c\textasciicircum 2 != 1}}
    Finally, we present an example for \Cref{Prop: c-boomerang unfiformity}.
    \begin{ex}\label[ex]{Ex: tight example q = 191}
        Let $q = 191$, let $f = X^3$, and let $a = 1$ and $b = 125$.
        Then, the LEX Gr\"obner basis $z >_{LEX} x$ of the $11$-Boomerang equations is
        \begin{align*}
            h_1 &= z + 126 x^{8} + 110 x^{7} + 95 x^{6} + 45 x^{5} + 65 x^{4} + 107 x^{3} + 15 x^{2} + 181 x + 166,
            \\
            h_2 &= (x + 102) \cdot (x + 140) \cdot (x + 145) \cdot (x + 167) \cdot (x + 173) \cdot (x + 179) \cdot (x + 184) \cdot (x + 45) \cdot (x + 73).
        \end{align*}
        Therefore, \Cref{Prop: c-boomerang unfiformity} is tight for $c = 11$.
        Also, note that $\gcd \left( 3, q - 1 \right) = 1$.
        The code for this example is provided in \Cref{Alg: tight example q = 191}.
    \end{ex}

    \section{Discussion}\label{Sec: Discussion}
    In this paper we have generalized the $c$-Boomerang uniformity bounds from \cite{Steiner-Boomerang} for permutation monomials (\Cref{Equ: c-BCT monomials}) to arbitrary permutation polynomials.
    In \cite[\S 5]{Steiner-Boomerang} the monomial bounds were utilized to estimate $c$-BCTs of permutations over $\Fqn$ based on the \emph{Generalized Triangular Dynamical System} (GTDS) \cite{SAC:RoySte24}.
    Likewise, our bounds can be utilized to generalize the estimations from \cite[Theorem 5.5]{Steiner-Boomerang}.

    In \cite[\S 4]{Steiner-Boomerang} tightness of the $\pm 1$-Boomerang uniformity bounds for monomials was investigated with symbolic Gr\"obner basis computations, and two conjectures for $\big( \LM_{DRL} (f) \; \big\vert \; f \in (F_1, F_2) \big) \subset \Fq [x, z]$, where $F_1$ and $F_2$ are the $\pm 1$-Boomerang polynomials, were proposed \cite[Conjecture~4.3, 4.4]{Steiner-Boomerang}.
    Our investigations into the arithmetic of $c$-Boomerang polynomial systems (\Cref{Sec: Arithmetic of Boomerang Polynomial Systems}) answer these conjectures in the affirmative if we have for the degree $d$ that $\gcd \left( d \cdot (d - 1), q \right) = 1$.
    In particular, if $q$ is a prime number and $d < q$, the conjectures are true.

    Readers might have noticed that most of our examples in \Cref{Sec: Constructing Tight Examples} become tight over relatively large field extension (compared to the base field).
    Therefore, one might wonder whether it is possible to quantify how small the fraction $\frac{d}{q}$ has to be  ``on average'' so that our bounds from \Cref{Th: degree bound c = 1}, \Cref{Prop: c-boomerang unfiformity} and \Cref{Th: degree bound c = -1} become tight.

    \appendix
    \section{\texorpdfstring{\SageMath}{SageMath} Code}
    \begin{algorithm}[H]
        \centering
        \caption{$c$-Boomerang uniformities of Dickson polynomials over $\F_{11}$.}
        \label{Alg: Dickson c-Boomerang uniformities}
\begin{lstlisting}[language=MyPython]
q = 11
K = GF(q)
sep = 30 * "-"
for c in range(1, q):
    print(sep)
    print("c:", c)
    print("alpha", "\t",
          "a", "\t",
          "b", "\t",
          "beta")
    for alpha in range(0, q):
        f = dickson_polynomial(7, K(alpha))
        a_max = 0
        b_max = 0
        D_max = 0
        for a in range(1, q):
            for b in range(1, q):
                polys = c_boomerang_polynomials(f, a, b, c=c)
                fes = [var**q - var for var in polys[0].parent().gens()]
                I = ideal(polys + fes)
                gb = I.groebner_basis()
                D = len(ideal(gb).normal_basis())
                if D > D_max:
                    a_max = a
                    b_max = b
                    D_max = D
        print(alpha, "\t",
            a_max, "\t",
            b_max, "\t",
            D_max)
\end{lstlisting}
    \end{algorithm}

    \begin{algorithm}[H]
    \centering
    \caption{Code for \Cref{Ex: tight example q = 11}.}
    \label{Alg: tight example q = 11}
\begin{lstlisting}[language=MyPython]
q = 11
K = GF(q)
P = PolynomialRing(K, "X")
X = P.gen()
f = X**7
a = 1
b = 3
polys = c_boomerang_polynomials(f, a, b)
I = ideal(polys)
gb = I.groebner_basis()
print("DRL Groebner Basis")
for poly in gb:
    print(poly)
print("")
gb_lex = ideal(gb).transformed_basis(algorithm="fglm")
print("Lexicographic Groebner Basis")
for poly in gb_lex:
    print(poly)
print("")
print("Factorization of Univariate LEX Polynomial")
for fac in gb_lex[-1].factor():
    print(fac)
\end{lstlisting}
    \end{algorithm}

    \begin{algorithm}[H]
    \centering
    \caption{Code for \Cref{Ex: tight example q = 16}.}
    \label{Alg: tight example q = 16}
\begin{lstlisting}[language=MyPython]
q = 2**4
Y = polygen(GF(2))
K = GF(q, "g", modulus=Y**4 + Y + 1)
g = K.gen()
f = dickson_polynomial(11, K(1))
a = g^2 + 1
b = g^3 + g^2 + g
polys = c_boomerang_polynomials(f, a, b)
polys = c_boomerang_polynomials(f, a, b)
I = ideal(polys)
gb = I.groebner_basis()

print("DRL Groebner Basis")
for poly in gb:
    print(poly)
print("")
gb_lex = ideal(gb).transformed_basis(algorithm="fglm")
print("Lexicographic Groebner Basis")
for poly in gb_lex:
    print(poly)
print("")
print("Factorization of Univariate LEX Polynomial")
for fac in gb_lex[-1].factor():
    print(fac)
\end{lstlisting}
    \end{algorithm}

    \begin{algorithm}[H]
        \centering
        \caption{Code for \Cref{Ex: tight example q = 257}.}
        \label{Alg: tight example q = 257}
\begin{lstlisting}[language=MyPython]
q = 257
K = GF(q)
P = PolynomialRing(K, "X")
X = P.gen()
f = X**5
a = 1
b = 48
polys = c_boomerang_polynomials(f, a, b, c=-1)
I = ideal(polys)
gb = I.groebner_basis()
print("DRL Groebner Basis")
for poly in gb:
    print(poly)
print("")
gb_lex = ideal(gb).transformed_basis(algorithm="fglm")
print("Lexicographic Groebner Basis")
for poly in gb_lex:
    print(poly)
print("")
print("Factorization of Univariate LEX Polynomial")
for fac in gb_lex[-1].factor():
    print(fac)
\end{lstlisting}
    \end{algorithm}

    \begin{algorithm}[H]
        \centering
        \caption{Code for \Cref{Ex: tight example q = 89}.}
        \label{Alg: tight example q = 89}
\begin{lstlisting}[language=MyPython]
q = 89
K = GF(q)
f = f = dickson_polynomial(7, K(1))
a = 1
b = 17
polys = c_boomerang_polynomials(f, a, b, c=-1)
I = ideal(polys)
gb = I.groebner_basis()
print("DRL Groebner Basis")
for poly in gb:
    print(poly)
print("")
gb_lex = ideal(gb).transformed_basis(algorithm="fglm")
print("Lexicographic Groebner Basis")
for poly in gb_lex:
    print(poly)
print("")
print("Factorization of Univariate LEX Polynomial")
for fac in gb_lex[-1].factor():
    print(fac)
\end{lstlisting}
    \end{algorithm}

    \begin{algorithm}[H]
        \centering
        \caption{Code for \Cref{Ex: tight example q = 191}.}
        \label{Alg: tight example q = 191}
\begin{lstlisting}[language=MyPython]
q = 191
K = GF(q)
P = PolynomialRing(K, "X")
X = P.gen()
f = X**3
a = 1
b = 125
polys = c_boomerang_polynomials(f, a, b, c=11)
I = ideal(polys)
gb = I.groebner_basis()
print("DRL Groebner Basis")
for poly in gb:
    print(poly)
print("")
gb_lex = ideal(gb).transformed_basis(algorithm="fglm")
print("Lexicographic Groebner Basis")
for poly in gb_lex:
    print(poly)
print("")
print("Factorization of Univariate LEX Polynomial")
for fac in gb_lex[-1].factor():
    print(fac)
\end{lstlisting}
    \end{algorithm}

    \printbibliography[heading=bibintoc]

@string{ieee =                  {IEEE}}

@string{springer =              "Springer"}

@string{acm =                   "Association for Computing Machinery"}

@book{Niederreiter-FiniteFields,
	author = {Lidl, Rudolf and Niederreiter, Harald},
	title = {Finite fields},
	publisher = {Cambridge Univ. Press},
	year = {1997},
	edition = {2nd},
	ISBN = {0-521-39231-4},
	series = {Encyclopedia of mathematics and its applications},
    address = {Cambridge},
}

@article{Ellingsen-cDifferential,
    author = {Ellingsen, P{\aa}l and Felke, Patrick and Riera, Constanza and St\u{a}nic\u{a}, Pantelimon and Tkachenko, Anton},
    journal = {IEEE Trans. Inf. Theory},
    title = {C-Differentials, Multiplicative Uniformity, and (Almost) Perfect c-Nonlinearity},
    year = {2020},
    volume = {66},
    number = {9},
    pages = {5781-5789},
    doi = {10.1109/TIT.2020.2971988}
}

@article{Stanica-cBoomerang,
    author = {St\u{a}nic\u{a}, Pantelimon},
    title = {Investigations on c-boomerang uniformity and perfect nonlinearity},
    journal = {Discrete Appl. Math},
    volume = {304},
    pages = {297-314},
    year = {2021},
    issn = {0166-218X},
    doi = {10.1016/j.dam.2021.08.002},
}

@article{Li-Boomerang-2024,
    author = {Li, Guanghui and Cao, Xiwang},
    title = {The c-boomerang uniformity of two classes of permutation polynomials over finite fields},
    journal = {Comput. Appl. Math.},
    year = {2024},
    month = {Oct},
    day = {09},
    volume = {43},
    number = {8},
    pages = {437},
    issn = {1807-0302},
    doi = {10.1007/s40314-024-02956-4},
}

@article{Li-Boomerang-2025,
    author = {Li, Guanghui and Cao, Xiwang},
    title = {The c-boomerang uniformity and c-boomerang spectrum of two classes of permutation polynomials over the finite field $\mathbb{F}_{2^{n}}$},
    journal = {Discrete Math.},
    volume = {348},
    number = {9},
    pages = {114543},
    year = {2025},
    issn = {0012-365X},
    doi = {10.1016/j.disc.2025.114543},
}

@book{Cox-Ideals,
    author = {Cox, David A. and Little, John and O'Shea, Donal},
    title = {Ideals, Varieties, and Algorithms: An Introduction to Computational Algebraic Geometry and Commutative Algebra},
    series = {Undergraduate Texts in Mathematics},
    year = {2015},
    publisher = {Springer International Publishing},
    isbn = {978-3-319-16720-6},
    doi = {10.1007/978-3-319-16721-3},
    edition = {4th},
}

@book{Kreuzer-CompAlg1,
    author = {Kreuzer, Martin and Robbiano, Lorenzo},
    title = {Computational Commutative Algebra 1},
    year = {2000},
    publisher = {Springer Berlin Heidelberg},
    address = {Berlin, Heidelberg},
    edition = {1st},
    isbn = {978-3-540-67733-8},
    doi = {10.1007/978-3-540-70628-1},
}

@book{Kreuzer-CompAlg2,
    author = {Kreuzer, Martin and Robbiano, Lorenzo},
    title = {Computational Commutative Algebra 2},
    year = {2005},
    publisher = {Springer Berlin Heidelberg},
    address = {Berlin, Heidelberg},
    edition = {1st},
    isbn = {978-3-540-28296-9},
    doi = {10.1007/3-540-28296-3},
}

@book{Goertz-AlgGeom,
    author = {G{\"o}rtz, Ulrich and Wedhorn, Torsten},
    title = {Algebraic Geometry {I}: Schemes},
    publisher = {Springer Fachmedien Wiesbaden},
    year = {2020},
    edition = {2nd},
    ISBN = {978-3-658-30732-5},
    DOI = {10.1007/978-3-658-30733-2},
    series = {Springer Studium Mathematik - Master},
}

@article{Stanica-SwappedInverse,
    author = {St\u{a}nic\u{a}, Pantelimon},
    title = {Low c-differential and c-boomerang uniformity of the swapped inverse function},
    journal = {Discrete Math.},
    volume = {344},
    number = {10},
    pages = {112543},
    year = {2021},
    issn = {0012-365X},
    doi = {10.1016/j.disc.2021.112543},
}

@article{Stanica-Characterization,
    author = {St\u{a}nic\u{a}, Pantelimon},
    title = {Using double {Weil} sums in finding the c-boomerang connectivity table for monomial functions on finite fields},
    journal = {Appl. Algebra Eng. Commun. Comput.},
    year = {2023},
    month = {Jul},
    day = {01},
    volume = {34},
    number = {4},
    pages = {581-602},
    issn = {1432-0622},
    doi = {10.1007/s00200-021-00520-9},
}

@article{Hasan-BinaryGold,
    author = {Hasan, Sartaj Ul and Pal, Mohit and St\u{a}nic\u{a}, Pantelimon},
    title = {The binary {Gold} function and its c-boomerang connectivity table},
    journal = {Cryptogr. Commun.},
    year = {2022},
    month = {Nov},
    day = {01},
    volume = {14},
    number = {6},
    pages = {1257-1280},
    issn = {1936-2455},
    doi = {10.1007/s12095-022-00573-8},
}

@article{Wang-Boomerang,
    author = {Wang, Yan-Ping and Wang, Qiang and Zhang, Wei-Guo},
    title = {Boomerang uniformity of normalized permutation polynomials of low degree},
    journal = {Appl. Algebra Eng. Commun. Comput.},
    year = {2020},
    month = {Jun},
    day = {01},
    volume = {31},
    number = {3},
    pages = {307-322},
    issn = {1432-0622},
    doi = {10.1007/s00200-020-00431-1},
}

@article{Gao-Irreducibility,
    author = {Gao, Shuhong},
    title = {Absolute Irreducibility of Polynomials via {N}ewton Polytopes},
    journal = {J. Algebra},
    volume = {237},
    number = {2},
    pages = {501-520},
    year = {2001},
    issn = {0021-8693},
    doi = {10.1006/jabr.2000.8586},
}

@book{Lidl-Dickson,
    author = {Lidl, Rudolf and Mullen, Gary L. and Turnwald, Gerhard},
    title = {Dickson polynomials},
    series = {Pitman monographs and surveys in pure and applied mathematics},
    year = {1993},
    edition = {1st},
    publisher = {Longman Scientific \& Technical},
    isbn = {0 582 09119 5},
    adress = {Essex, United Kingdom},
}

@article{Steiner-Boomerang,
	author = {Steiner, Matthias Johann},
	title = {A degree bound for the c-boomerang uniformity of permutation monomials},
	journal = {Appl. Algebra Eng. Commun. Comput.},
	year = {2024},
	month = {Oct},
	issn = {1432-0622},
	doi = {10.1007/s00200-024-00670-6},
}

@phdthesis{Buchberger,
    title = {Ein Algorithmus zum Auffinden der Basiselemente des Restklassenringes nach einem nulldimensionalen Polynomideal},
    author = {Buchberger, Bruno},
    school = {Universit\"{a}t Innsbruck},
    year = {1965}
}

@inproceedings{Naehring-Practical,
    author = {Naehrig, Michael and Lauter, Kristin and Vaikuntanathan, Vinod},
    title = {Can homomorphic encryption be practical?},
    booktitle = {Proceedings of the 3rd ACM Workshop on Cloud Computing Security Workshop},
    year = {2011},
    isbn = {9781450310048},
    publisher = {Association for Computing Machinery},
    address = {New York, NY, USA},
    doi = {10.1145/2046660.2046682},
    pages = {113–124},
    numpages = {12},
    location = {Chicago, Illinois, USA},
    series = {CCSW '11}
}

@book{Hartshorn-Euclid,
    author = {Hartshorne, Robin},
    title = {Geometry: {E}uclid and Beyond},
    year = {2000},
    publisher = {Springer New York},
    address = {New York, NY},
    isbn = {978-0-387-22676-7},
    doi = {10.1007/978-0-387-22676-7},
    edition = {1},
}

@book{Kemper-CommutativeAlgebra,
    author = {Kemper, Gregor},
    title = {A Course in Commutative Algebra},
    year = {2011},
    publisher = {Springer Berlin Heidelberg},
    address = {Berlin, Heidelberg},
    isbn = {978-3-642-03545-6},
    doi = {10.1007/978-3-642-03545-6},
    edition = {1},
}

@article{Faugere-FGLM,
    author = {Faug\`{e}re, Jean-Charles and Gianni, Patrizia and Lazard, Daniel and Mora, Teo},
    title = {Efficient Computation of Zero-dimensional Gröbner Bases by Change of Ordering},
    journal = {Journal of Symbolic Computation},
    volume = {16},
    number = {4},
    pages = {329-344},
    year = {1993},
    issn = {0747-7171},
    doi = {10.1006/jsco.1993.1051},
}

@misc{SageMath,
    key = {Sage25},
    author = {{The Sage Developers}},
    title = {{S}age{M}ath, the {S}age {M}athematics {S}oftware {S}ystem ({V}ersion 10.7)},
    url = {https://www.sagemath.org},
    year = {2025},
}

\end{document}